\documentclass[12pt,oneside]{amsart}
\usepackage{amssymb,amsmath}
\usepackage{a4,typearea}
\usepackage{stmaryrd}
\usepackage{hyperref}
\usepackage[all]{xy}

\def\usexypic{1}

\newcounter{lemmacounter}
\newcounter{thmcounter}

\newcounter{propcounter}
\newcounter{corcounter}

\numberwithin{lemmacounter}{section}
\numberwithin{propcounter}{section}
\numberwithin{equation}{section}

\newtheorem{lemma}[lemmacounter]{Lemma}
\newtheorem{proposition}[propcounter]{Proposition}
\newtheorem*{corollary*}{Corollary}
\newtheorem{corollary}[corcounter]{Corollary}

\newtheorem{theorem}[thmcounter]{Theorem}
\newtheorem*{theorem*}{Theorem}

\newcommand{\loccitd}{\hbox{\it loc.\,cit. }}

\newcommand{\gal}[1]{{\rm Gal}({#1})}
\newcommand{\ram}[2]{{G}_{#2}({#1})}

\newcommand{\gl}[2]{{\rm GL}_{#2}({#1})}

\newcommand{\mat}[2]{{\rm Mat}_{#2}({#1})}

\newcommand{\IG}{{\bf G}}
\newcommand{\IC}{{\bf C}}
\newcommand{\IR}{{\bf R}}
\newcommand{\IQbar}{\overline{\bf Q}}
\newcommand{\IZ}{{\bf Z}}
\newcommand{\IN}{{\bf N}}

\newcommand{\IF}{{\bf F}}

\newcommand{\IQ}{\mathbf{Q}}
\newcommand{\IQpbar}{\overline{\mathbf{Q}}_p}
\newcommand{\IQp}{\mathbf{Q}_p}
\newcommand{\IQpunr}{\mathbf{Q}_p^{\rm unr}}
\newcommand{\tors}[1]{#1_{\rm tors}}

\newcommand{\heightS}{h}
\newcommand{\height}[1]{\heightS({#1})}
\newcommand{\ssm}{\smallsetminus}
\newcommand{\aut}[1]{{{\rm Aut\,}{#1}}}
\newcommand{\autc}[2]{{{\rm Aut}_{#1}{#2}}}

\newcommand{\ab}[1]{{#1}^{\rm ab}}
\newcommand{\tr}[1]{{#1}^{\rm mr}}

\renewcommand{\O}{\mathcal{O}}
\newcommand{\imag}[1]{{\rm Im}({#1})}

\newcommand{\atopx}[2]{\genfrac{}{}{0pt}{}{#1}{#2}}

\newcommand{\T}[2]{T_{#2}({#1})}
\begin{document}


\author{P. Habegger}
\title[ Small Height and Nonabelian Extensions]{ 
{  Small Height and 
  Infinite Nonabelian  Extensions} \\
}


\maketitle

\begin{abstract}
Let $E$ be an elliptic curve defined over $\IQ$ without complex
multiplication. 
The field $F$ generated over $\IQ$ by all torsion points of $E$
is an infinite, nonabelian Galois extension of the rationals which
has unbounded, wild ramification above all primes. We prove
that the absolute logarithmic Weil height of an element of
$F$ is either zero or bounded from below  by a positive
constant depending only on $E$. We also show that the N\'eron-Tate
height has a similar gap on $E(F)$
and use this to determine the structure of the group $E(F)$. 
\end{abstract}

\section{Introduction}

By Northcott's Theorem, there are only finitely many algebraic numbers
of  bounded degree and bounded absolute logarithmic Weil height, or
short, height. 
This height and the relevant properties are covered in greater detail
in Section \ref{sec:height}.
Kronecker's Theorem states that
an algebraic number has height zero if and only if it is zero or a
root of unity. 
So any non-zero element of a number field that is
not a root of unity  has 
height bounded from below uniformly by a positive real number.

A field that is algebraic (but not necessarily of finite degree)
over $\IQ$ 
is said to satisfy the Bogomolov property if zero is  isolated among its height values.
The property's name was
 motivated by the eponymous conjecture
on points of small N\'eron-Tate
 height on curves of genus at least $2$
and appears in work of Bombieri and Zannier \cite{BZ:infinitebogo}.

The fundamental example $\height{2^{1/n}} = (\log 2)/n$ shows that 
$\IQ(2^{1/2},2^{1/3},\ldots)$, and so in particular the field of algebraic
numbers, does not satisfy the Bogomolov property.
But there are many infinite  extensions  which do and we
 will mention some known
examples after stating our main results. 

In this paper we first exhibit a new class  of
 infinite,  nonabelian
Galois extensions of $\IQ$ satisfying
 the Bogomolov property. These will be
related to an elliptic curve $E$ defined over $\IQ$.
We let $\tors{E}$ denote the group of
torsion points of $E$ defined over an algebraic closure  
 of $\IQ$. The field $\IQ(\tors{E})$ is generated by the set of $x$-
 and $y$-coordinates of the points in
 $\tors{E}$ with respect to a  Weierstrass model
 of $E$ with rational coefficients. 

 \begin{theorem}
   \label{thm:main}
Suppose $E$ is an elliptic curve defined over $\IQ$.
Then $\IQ(\tors{E})$ satisfies the Bogomolov property. 
 \end{theorem}

The N\'eron-Tate height is
 a natural height function defined on the algebraic points of the
 elliptic curve $E$  itself, we will review its definition
 in Section \ref{sec:ntheight}.
The analog of Northcott's Theorem holds;
in other words $E$ contains only finite many points 
 of bounded degree and bounded N\'eron-Tate height. 
  Kronecker's
Theorem for the N\'eron-Tate height is also true
 since $\hat h$ vanishes precisely on the torsion points
of $E$. 

The second result of this paper is the elliptic  analog of Theorem
\ref{thm:main} and bounds from below the N\'eron-Tate height.
It gives an affirmation answer to a question of
 Baker \cite{Baker:CMBogomolov} for elliptic curves defined over $\IQ$.

\begin{theorem}
\label{thm:main_ec}
Suppose $E$ is an elliptic curve defined over $\IQ$. 
 There exists
  $\epsilon > 0$ such that if $A\in E(\IQ(\tors{E}))$ is non-torsion, then
 $  \hat h(A) \ge \epsilon$.
\end{theorem}

We now discuss how  our results are related to  the literature.
 Amoroso and Dvornicich \cite{AD:abelianbogo} proved that all
 abelian extensions of $\IQ$ satisfy the Bogomolov property
 thus affirming a question raised by Bombieri and Zannier. 
This result covers the field generated by all roots of unity. 
Later Amoroso and Zannier proved  a more precise  height lower bound 
\cite{AZ:relDobrowolski} in the spirit of  Lehmer's question. 
A special case of their result implies that the maximal abelian extension
 $\ab{K}$ of a number field $K$ satisfies the Bogomolov property. 
This statement was later refined  by the same authors
\cite{AZ:unifrel} to yield a uniform lower bound that depends only
on the degree $[K:\IQ]$. 

We say that an elliptic curve defined over a field of characteristic
zero has complex multiplication if it has a non-trivial endomorphism
defined over an algebraic closure of the base field. 

On the elliptic side, Baker \cite{Baker:CMBogomolov} proved that if
$E$ is defined over $K$ and either has complex
multiplication or non-integral $j$-invariant, then a point in $E(\ab{K})$ cannot have arbitrarily
small positive N\'eron-Tate height. 
Silverman \cite{Silverman:abelian} proved the same conclusion 
with no restriction on $E$. 

If $E$ has complex multiplication and if all
endomorphisms of $E$ are  defined over $K$, then $K(\tors{E})$ is
an infinite abelian extension of $K$.
In other words $K(\tors{E})\subset \ab{K}$. 
 Amoroso and Zannier's result implies that
 $\IQ(\tors{E})$ satisfies the Bogomolov property as 
this property 
 is  clearly inherited by  subfields.

So for $K=\IQ$ we recover 
 Theorem \ref{thm:main} if $E$ has complex
multiplication. Under the same assumption on $E$,
 Baker's result implies Theorem
\ref{thm:main_ec}. 

Our results, however, hold when
 $E$  does not have complex
multiplication and is defined over $\IQ$. In this case
$\IQ(\tors{E})$ is still a Galois extension of $\IQ$. But it is never
abelian as we will see in a moment. 
The Galois group of this extension is sufficiently anabelian to push
Theorems \ref{thm:main} and \ref{thm:main_ec}
  outside the immediate range of  earlier
results involving abelian extensions. 
Indeed, if we were to assume $\IQ(\tors{E}) \subset \ab{K}$
for some number field $K$,
then $\gal{\IQ(\tors{E})/\IQ}$ would contain 
the abelian subgroup $\gal{\IQ(\tors{E})/K\cap
  \IQ(\tors{E})}\cong \gal{K(\tors{E})/K}$ with
 index bounded by $d = [K:\IQ]$.
 For an integer $N\ge 1$ we let  $E[N]\subset \tors{E}$ denote the
 subgroup of points of order dividing $N$.
It is isomorphic to $(\IZ/N\IZ)^2$.
By Serre's Theorem \cite{Serre:Galois} there exists a prime $p > d$
such that the natural Galois representation
$\gal{\IQ(\tors{E})/\IQ}\rightarrow \aut E[p]$  is surjective. 
We  fix an isomorphism $\aut E[p]\cong \gl{\IF_p}{2}$
and conclude that $\gl{\IF_p}{2}$ contains an abelian subgroup of
index at most $d$. By group theory, the $d!$-th power of a matrix
in $\gl{\IF_p}{2}$ lies in said abelian subgroup. In particular, 
the matrices
\begin{equation*}
\left(
  \begin{array}{cc}
    1 & d! \\ 0 & 1
  \end{array}
\right)
\quad\text{and}\quad
\left(
  \begin{array}{cc}
    1 & 0 \\ d! & 1
  \end{array}
\right)
\end{equation*}
commute in $\gl{\IF_p}{2}$. This is absurd because $p \nmid d!$.

Height lower bounds are not only available for  extensions with
certain Galois groups,  but also for  fields 
satisfying  a local restriction.
 The early
result of Schinzel \cite{Schinzel:ProdConj}
 implies that $\tr{\IQ}$,
 the maximal totally real extension of $\IQ$,
satisfies the Bogomolov property. 
The Weil pairing is compatible with the action of the Galois group. 
From this we find that $\IQ(E[N])$ contains a primitive
$N$-th root of unity. So $\IQ(E[N])$ cannot be
contained in a totally real number field if $N\ge 3$.

Zhang \cite{Zhang:equi} proved the analog of Schinzel's result for
abelian varieties. In  our case it states that
 $E(\tr{\IQ})$  contains only finitely many torsion points
and does not 
 contain points of arbitrarily
small positive N\'eron-Tate height.  Zhang
 deduced
the same consequence for finite extensions of
 $\tr{\IQ}$. Of course, $E(\IQ(\tors{E}))$ contains infinitely many torsion
points. So 
 $\IQ(\tors{E})$ is not a finite extension of a 
 totally real extension of $\IQ$.

 Bombieri and Zannier \cite{BZ:infinitebogo} studied an
 analog of Schinzel's result where the $p$-adic numbers $\IQ_p$
replace the reals.  
They discovered that any normal algebraic extension of $\IQ$ which admits an
embedding into a finite extension $L$ of  $\IQ_p$ satisfies the Bogomolov property. 
Our field $\IQ(\tors{E})$ cannot lie in such an $L$, even if $E$
is allowed to have complex multiplication. Indeed, otherwise we would have
$\IQ(E[p^n]) \subset L$ for all positive integers $n$. 
As above we see that $\IQ(E[p^n])$ contains  a primitive
$p^n$-th root of unity $\zeta$. It is known that
 $\IQ_p(\zeta) / \IQ_p$ 
has degree $p^{n-1}(p-1)$. 
So  $\IQ_p(\zeta)\subset L$ is impossible for  $n$
sufficiently large.
Baker
and Petsche \cite{BP:globaldiscrepancy} proved the analog of
 Bombieri and Zannier's Theorem for elliptic curves.

Widmer \cite{Widmer:Northcott} has presented a sufficient condition
for a field to contain only finitely many elements of bounded height. 
This property is stronger than the Bogomolov property and is not shared
by $\IQ(\tors{E})$ or even $\ab{\IQ}$. 
Very recently, Amoroso,  David, and Zannier \cite{ADZ}  
proved a common generalization of  the first and third's result
\cite{AZ:unifrel} and some aspects of
 Bombieri and Zannier's Theorem.

Zograf \cite{zograf}
 and independently Abramovich \cite{abramovich99}
proved a lower bound for the gonality in a tower of 
classical  modular curves in characteristic zero.
Their lower  bound
is  linear in terms
 of the degree of the morphism to the base of the tower.
This corresponds precisely to the Bogomolov property transposed to
extensions of a fixed function field.
In more recent work Ellenberg, Hall, and Kowalski \cite{EHK}
deduced lower bounds for the gonality in more general  towers of
curves.
Poonen provided good evidence that 
  a linear lower bound for the gonality
also holds for  towers coming from modular curves in 
 positive characteristic  \cite{Poonen:gonality}.
In all these results, the  corresponding  function field extensions are
usually not abelian. It is remarkable that non-commutativity is an
obstacle on the number field side but is  necessary (though not
sufficient) 
on the function field side 
 to ensure that certain graphs
are almost expanders 
\cite{EHK}.


We now give an overview of our proof of Theorem \ref{thm:main}. 
For this let us suppose  that $E$ does not have complex
multiplication. 
Our argument  uses
the  decomposition of the height into local terms. 
Say $N\ge 1$ is an integer.
The basic idea is to use two  metric estimates, one
non-Archimedean and one Archimedean, 
in the number field $\IQ(E[N])$.
We will use both to derive a height lower bound on this field that is
independent of $N$.

The  non-Archimedean estimate is 
done at places above an auxiliary prime number
$p$. 
Elkies \cite{Elkies} proved that $E$ has  supersingular
reduction at infinitely many primes. 
It will
suffice to work with only one  $p$. However we must 
 arrange, among other things, that the   representation
 $\gal{\IQ(\tors{E})/\IQ}\rightarrow\aut E[p]$ is surjective.
By Serre's
 Theorem this is true for all but finitely many $p$.
At the moment,  the Theorem of Elkies is not known for elliptic curves
over a general number field.  So we restrict ourselves to 
elliptic curves 
defined over $\IQ$. 

The prime $p$ is fixed once and for all in terms of $E$ and
does not depend of $N$.
Our approach is based on
 studying the   representations
\begin{equation}
\label{eq:galoisrep}
  \gal{\IQ_p(E[\ell^n])/\IQ_p} \rightarrow \aut E[\ell^n]
\quad\text{for an integer}\quad n\ge 1
\end{equation}
as $\ell$ varies over the prime divisors of $N$, including
 $\ell=p$ if necessary. 

No ramification occurs when $\ell\not = p$.
In this case we will obtain an explicit  height lower bound
swiftly using the product formula in Lemma \ref{lem:primetoptorsion}. 
The crucial point is that supersingularity forces the square of the
Frobenius to act as a scalar on the reduction of $E$ modulo $p$.
A lift of 
this square
to characteristic zero is in the
center of Galois group of $\IQ(E[\ell^n])/\IQ$, a  fact
that makes up for the
lack of commutativity.

 Ramification occurs when $\ell=p$ and here lies
the main difficulty in proving Theorem \ref{thm:main}.
We will 
 describe  representations (\ref{eq:galoisrep}) using Lubin-Tate
 modules. 
Again we need 
that 
$E$ has good supersingular
reduction at $p$.
But we can no longer rely on  Frobenius and instead use Lubin-Tate theory
to find  a suitable replacement inside
a higher ramification group. In general 
this substitute does not lie in
the center of the Galois group. But its 
 centralizer turns out to be  sufficiently large for our purposes.

A dichotomy into an unramified and a ramified case already 
appeared in the original work  of Amoroso-Dvornicich  on
abelian extensions of $\IQ$.
But in the ramified case, our non-Archimedean estimate is
 significantly weaker
when compared to the unramified case.
It  cannot be used with the product formula 
to deduce 
 Theorem \ref{thm:main} directly. 
The reason is 
 described in greater detail in 
 the beginning of Section \ref{sec:equi}.
We  remedy this deficiency by treating
the Archimedean places more carefully. 
Thus  our second  estimate is Archimedean and 
 relies on Bilu's Equidistribution Theorem
 \cite{Bilu} for algebraic numbers of small height.

Bilu's Theorem has other ties to our problem as it  yields
 another proof of
Schinzel's Theorem that $\tr{\IQ}$ has the Bogomolov property.
Supersingular reduction was also used in Bombieri and Zannier's
work \cite{BZ:96} on the Bogomolov Conjecture for subvarieties of
abelian varieties.

The proof of Theorem \ref{thm:main_ec} follows along similar lines
as the proof of Theorem \ref{thm:main}. 
We  rely on a  decomposition of the N\'eron-Tate height into local
height functions. 
And we also split 
the non-Archimedean local estimates
up into an unramified and a ramified case.
In the elliptic setting, Szpiro, Ullmo, and Zhang's 
 Equidistribution Theorem  \cite{SUZ:equi}
  substitutes  Bilu's result. 
However, local terms in the N\'eron-Tate
height, unlike the local terms in the Weil height,  can take negative
values at  non-Archimedean places. 
So we will also need 
 a theorem of Chambert-Loir \cite{CL:mesures}
which yields non-Archimedean equidistribution at places of split
multiplicative reduction. 
Alternatively, 
  Baker and Petsche's \cite{BP:globaldiscrepancy} 
simultaneous approach to 
  Archimedean and non-Archimedean equidistribution 
can also be used.

Theorems \ref{thm:main} and its elliptic counterpart Theorem
\ref{thm:main_ec} have a common  reformulation in terms of the
split semi-abelian variety 
$S=\IG_m \times E$. 
A natural N\'eron-Tate height on $S(\IQbar)$
is  
given by
 $\hat h(\alpha,A) = \height{\alpha}+\hat h(A)$
for $\alpha\in \IG_m(\IQbar)$ and $A\in E(\IQbar)$.
Then
 $\hat h$ vanishes precisely on $\tors{S}$, the group of all 
torsion points of $S$.
As we have already seen, the Weil pairing implies  $\IQ(\tors{S}) =
\IQ(\tors{E})$. 
Our two 
previous theorems are repackaged in the following corollary whose
proof is immediate.

\begin{corollary}
Suppose $E$ is an elliptic curve defined over $\IQ$
and let $S = \IG_m\times E$.
 There exists
  $\epsilon > 0$ such that if $P\in S(\IQ(\tors{S}))$ is
 non-torsion, then 
$\hat h(P) \ge \epsilon$.
\end{corollary}

Let us state some open questions and problems related to our results.

By Theorem \ref{thm:main} there  exists  $\epsilon>0$, depending on
$E$, such that for 
any non-zero $\alpha\in\IQ(\tors{E})$ that is not a
root of unity we have $\height{\alpha}\ge \epsilon$. 
It is a  natural problem to determine an explicit $\epsilon$ in terms of
the coefficients of a minimal Weierstrass equation of $E$. 
This problem is amenable to our method given explicit versions
of the theorems of Bilu, Elkies and Serre.  
But an effective version of the Theorem of Elkies is likely to  
  introduce quantities depending on $E$. 
On the other hand, the author was unable to find an $E$ and $\alpha$
such that $\height{\alpha}$ is positive but arbitrarily small. 
Can one choose $\epsilon$, implicit in Theorem
\ref{thm:main},  to be 
independent of $E$? A similar question can be raised in the context
of Theorem \ref{thm:main_ec}.

Do Theorems \ref{thm:main} and
\ref{thm:main_ec} hold with $\IQ(\tors{E})$ replaced by a finite
extension?  Say $\epsilon  > 0$. According to
 a conjecture of David, formulated for 
 abelian varieties defined over number fields, there should exist a constant
$c>0$ depending only on $E$ and $\epsilon$ with
\begin{equation}
\label{eq:reldobrowolski}
  \hat h(A) \ge \frac{c}{[\IQ(\tors{E})(A):\IQ(\tors{E})]^{1+\epsilon}}
\end{equation}
for all algebraic points $A$ of $E$ that are not torsion. 
This is a  so-called relative Dobrowolski-type inequality.
It is even expected to hold for $\epsilon=0$. 
Ratazzi \cite{Ratazzi:04} proved the generalization 
of (\ref{eq:reldobrowolski}) 
to elliptic curves with complex multiplication 
defined over a number field.
Proving inequality (\ref{eq:reldobrowolski})
for elliptic curves without complex multiplication is a longstanding
open problem, even with $\IQ(\tors{E})$  replaced by $\IQ$. 
Variants of such  estimates
 have interesting
applications to  unlikely intersections on
abelian varieties and algebraic tori
 \cite{CL:bourbaki}. 

Suppose $E'$ is a second elliptic curve defined over $\IQ$
and let $F = \IQ(\tors{E},\tors{E'})$. Then 
David's Conjecture  for 
 the abelian surface $E\times E'$ expects that
$\hat h(A)+\hat h(A')$ is bounded from below by a positive constant
if at least one among $A\in E(F),A'\in E'(F)$
 is not torsion. In a similar vein 
we  ask if the field
 $F$ satisfies the Bogomolov property.

The N\'eron-Tate height  plays an
important role in the proof of the Mordell-Weil Theorem. 
Indeed, the famous descent argument  relies on
the basic property that an elliptic curve contains  only finitely
 points 
defined over a fixed number field with bounded
 height. This finiteness property is stronger than what we proved
in  Theorem \ref{thm:main_ec} for the group $E(\IQ(\tors{E}))$. However, 
our result has an amusing consequence for the structure of this group.
 The following corollary uses a  
group theoretic result of Zorzitto \cite{Zorzitto}.

\begin{corollary}
\label{cor:mwgroup}
Suppose $E$ is an elliptic curve defined over $\IQ$. 
Then $E(\IQ(\tors{E}))/\tors{E}$ is a free abelian group of countable
infinite rank. In other words
\begin{equation*}
  E(\IQ(\tors{E})) \cong \tors{E} \oplus \bigoplus_{\IN}\IZ. 
\end{equation*}
\end{corollary}

Frey and Jarden \cite{FreyJarden} showed
 that   the smaller group $E(\ab{\IQ})$ already has infinite
 rank. That is, it contains an infinite sequence of elements that do not
 satisfy a non-trivial linear relation involving finitely many integer
 coefficients.
Our contribution  is in showing that $E(\IQ(\tors{E})) / \tors{E}$ is a free
abelian group.

We briefly discuss how this paper is organized. 
Section \ref{sec:prelim}  deals
mainly  with
issues of notation. In Section \ref{sec:supersingular} we review
the implications of Lubin-Tate theory for 
 the Galois representation (\ref{eq:galoisrep}) when $\ell=p$.
The  local non-Archimedean estimates used in the proof of Theorem \ref{thm:main} are derived in Section \ref{sec:localmetric}.
 In Section \ref{sec:global} we obtain a
 preliminary height lower
 bound in direction of Theorem \ref{thm:main}. It
is  then refined in Section \ref{sec:descending} using a Kummerian
 descent argument. 
Bilu's Equidistribution Theorem then 
completes the  
 proof that $\IQ(\tors{E})$ satisfies the Bogomolov property
  in Section
 \ref{sec:equi}. In Section \ref{sec:thm2} we
 turn our attention to lower bounds for the N\'eron-Tate height. The first half of
 this section contains a review of the N\'eron-Tate height while the
 second half finalizes the proof of Theorem \ref{thm:main_ec}
and contains the proof of Corollary \ref{cor:mwgroup}. 

This work was initiated at the 
 Institute for Advanced Studies in Princeton.
 I thank Enrico Bombieri for the invitation and the
 Institute for its hospitality and financial support. 
I am in dept to Jeff Vaaler and Martin Widmer for 
pointing out Zorzitto's result and its relevance to small
height.
It is my pleasure to thank  Jordan Ellenberg,
 Florian Herzig, and Lars K\"uhne for
discussions, comments, and references. 
I am  grateful to 
 Laurent Berger, Clemens Fuchs, Andrew Kresch, and 
Gisbert W\"ustholz for organizing a
 Workshop on $p$-adic Periods in Alpbach, Austria in  Summer 2010.
Many things I learned there proved
influential for the current work. 
Finally, I would like to thank the referees for remarks that helped 
 improve the exposition and simplify certain arguments, e.g. Lemma
 \ref{lem:qpower} and its elliptic  counterpart Lemma \ref{lem:kummer_ec}. 
This research was partially supported by
SNSF project
number 124737.

\section{Preliminaries on Heights and  Local Fields}
\label{sec:prelim}

The group of units of a ring $R$ is denoted by $R^\times$. The natural
numbers $\IN$ are $\{1,2,3,\ldots\}$.

\subsection{Heights}
\label{sec:height}

Let $K$ be a number field. A place $v$ of $K$ is an absolute
value  $|\cdot|_v:K\rightarrow [0,\infty)$ 
whose restriction $w$ to $\IQ$ is either the standard complex
absolute value $w=\infty$ or $w=p$, the
 $p$-adic absolute value for a prime $p$.
In the former case we write $v|\infty$ and call $v$ infinite or Archimedean. 
In the latter case we write $v|p$ or $v\nmid\infty$ and call $v$
finite
or non-Archimedean.
A place is  finite if and only if it
 satisfies the ultrametric triangle
inequality.
The completion of $K$ with respect to $v$ is denoted with $K_v$.
We use the same symbol $|\cdot|_v$ for the absolute value on $K_v$.
The set of finite places can be identified naturally with the set of
non-zero prime ideals of the ring of integers of $K$. 
The infinite places are in bijection with field embeddings
$K\rightarrow \IC$ up to complex conjugation.
We define the local degree of $v$ as $d_v=[K_v:\IQ_{w}]$. They satisfy
\begin{equation}
\label{eq:localdegrees}
 \sum_{v|w}d_v= [K:\IQ],
\end{equation}
cf. Chapter 1.3 \cite{BG}.


The absolute logarithmic Weil height, or short height, 
of $\alpha\in K$ is defined to be
\begin{equation}
\label{eq:heightdef}
\height{\alpha}=  \frac{1}{[K:\IQ]}\sum_{v} d_v \log \max\{1,|\alpha|_v\}
\end{equation}
where $v$ runs over all places of $K$.

It is well-known that the height does not change if $K$ is replaced by
another  number field containing $\alpha$. Hence we have a
well-defined  
function $\heightS$ with domain any algebraic closure of
$\IQ$ taking non-negative real values. 
Kronecker's Theorem states that $\height{\alpha}$ vanishes precisely
when $\alpha=0$ or $\alpha$ is a root of unity. 
For these two statements we refer to Chapter 1.5 \cite{BG}.

We list some properties of our height which we will refer to as basic
height properties in the following.
Our definition (\ref{eq:heightdef}) implies 
\begin{equation}
\label{eq:basicheight}
  \height{\alpha\beta}\le \height{\alpha}+\height{\beta}
\quad\text{and}\quad \height{\alpha^k} = k\height{\alpha}
\end{equation}
if $\beta\in K$ and $k\in\IN$. 
The so-called product formula
\begin{equation*}
  \sum_{v} d_v \log |\alpha|_v = 0
\end{equation*}
holds  if $\alpha\not = 0$; it is
proved in Chapter 1.4 \cite{BG}.  One consequence is $\height{\alpha} =
\height{\alpha^{-1}}$. Combining this equality with
(\ref{eq:basicheight}) we deduce
\begin{equation*}
  \height{\alpha^k} = |k|\height{\alpha}
\quad\text{if}\quad\alpha\not=0 \quad \text{and}\quad k\in\IZ.
\end{equation*}
If $\zeta\in K$ is a root of unity,
then  $|\zeta|_v =1$ for all places $v$ of $K$. Hence
$\height{\zeta}=0$ and more generally
\begin{equation*}
  \height{\zeta\alpha}=\height{\alpha}. 
\end{equation*}
Finally, if $\alpha'$ is a conjugate of $\alpha$ over $\IQ$,  then
$\height{\alpha'}=\height{\alpha}$. 

\subsection{Local Fields}
\label{ssec:local}
If $K$ is a valued field, then $\O_K$ denotes its ring of integers
and $k_K$  its residue field.
Say $K/F$ is a finite Galois extension of discretely valued fields. We
shall assume that the valuation on $F$ is non-trivial.
Let $w : K \rightarrow \IZ \cup \{+\infty\}$ denote the surjective valuation. 
If $i \ge -1$ then
\begin{equation*}
  \ram{K/F}{i} = 
\{ \sigma\in \gal{K/F};\,\, 
w(\sigma(a) - a) \ge i + 1 \text{ for all } 
a \in \O_K \}
\end{equation*}
is the $i$-th higher ramification group of $K/F$.
We get a  filtration
\begin{equation*}
  \gal{K/F} = \ram{K/F}{-1} \supset 
\ram{K/F}{0} \supset  \ram{K/F}{1}\supset \cdots
\end{equation*}
where $\ram{K/F}{0}$ is the inertia group of $K/F$.

Let $p$ be a prime and let
 $\IQ_p$ be the field of $p$-adic numbers with 
absolute value $|\cdot|_p$. 
The prime $p$ will be fixed throughout the proof of our two 
theorems.
We will work with a
 fixed  algebraic closure $\IQpbar$ of  $\IQp$
and extend $|\cdot|_p$ to $\IQpbar$. 
All
algebraic extensions of $\IQ_p$ will  be  subfields of $\IQpbar$.

If $f\in\IN$ we call
\begin{equation*}
  \IQ_{p^f} \text{ the unique unramified extension of degree $f$ of $\IQ_p$ inside $\IQpbar$.} 
\end{equation*}
The integers in $\IQ_{p^f}$ will also be denoted by $\IZ_{p^f}$.
The union of all $\IQ_{p^f}$ is $\IQpunr$, the maximal unramified
extension of $\IQ_p$ inside $\IQpbar$. 
We let
$\varphi_p \in \gal{\IQpunr/\IQ_p}$
denote the lift  of the Frobenius automorphism.
We write $\varphi_{p^f} = \varphi_p^f$.

 For definiteness we
let $\IQbar$ denote the algebraic closure of $\IQ$ in $\IQpbar$.
We will  consider number fields to be subfields of
$\IQbar$ and hence of $\IQpbar$. 
Say $K$ is a finite  extension of $\IQ$. Then $|\cdot|_p$ restricts to
a finite place of $v$ of $K$. The completion
 $K_v$ can be taken to be the topological closure of $K$ inside  $\IQpbar$.
So if $K$ is a Galois extension of $\IQ$ one can  identify
$\gal{K_v/\IQ_p}$  with a subgroup of $\gal{K/\IQ}$ by restricting.

If $n\ge 0$ then
$\mu_{p^n}\subset\IQbar$ 
denotes the group of roots of unity
 with order dividing $p^n$.
Let $\mu_{p^\infty} \subset \IQbar$ denote the group of
roots of unity whose orders are a power of $p$.
Hence $\mu_{p^\infty}$ is the union of all $\mu_{p^n}$. 
We write $\mu_{\infty}$ for all roots of unity in $\IQbar$.

We collect some basic, but useful, facts on finite extensions of the $p$-adics.

\begin{lemma}
\label{lem:lift}
Let $F\subset\IQpbar$ be a finite extension of $\IQ_{p}$.
  Let $K,L \subset \IQpbar$ be
 finite Galois extensions of $F$ with
$K/F$ totally ramified and $L/F$ unramified.
  \begin{enumerate}
\item[(i)] We have $K\cap L = F$ and 
  \begin{equation*}
    \gal{KL/F}\ni \sigma \mapsto (\sigma|_K,\sigma|_L)
\in \gal{K/F}\times\gal{L/F}
  \end{equation*}
is an isomorphism of groups.
  \item [(ii)]
The extension
$KL/K$ is unramified of degree $[L:F]$, and the extension $KL/L$ is
totally ramified of degree $[K:F]$.
\item[(iii)]
Say $i\ge -1$.
If $\sigma \in \gal{KL/L}\cap G_i(KL/F)$ then $\sigma|_K\in
\ram{K/F}{i}$. Moreover, 
the induced map
$\gal{KL/L}\cap G_i(KL/F)\rightarrow \ram{K/F}{i}$ is an isomorphism
of groups.
  \end{enumerate}
\end{lemma}
\begin{proof}
The extension $(K\cap L)/F$ is totally ramified and unramified.
 A totally ramified and unramified extension of local
fields is trivial. So $K\cap L=F$.
The second claim of part (i) is now a basic result of Galois theory.

To prove part (ii) we can use part (i) to conclude
that  $KL/L$ is Galois with group isomorphic to 
$\gal{K/F}$. In particular, $KL/L$ is an
extension of degree $e=[K:F]$.
By a similar argument, $KL/K$ is of degree $f=[L:F]$. 
We remark that $KL/L$ and $KL/F$ have the same ramification index $e'$
since $L/F$ is unramified. 
In particular, $e'\ge e$. On the other hand, $e' \le [KL:L]=e$. So
$e'=e$ and thus $KL/L$ is totally ramified. We also conclude that
$KL/K$ is unramified. So part (ii) holds.

Let $\pi \in \O_K$ be a uniformizer for  $K$.
Moreover, let $x_1,\dots,x_f \in \O_L$ be lifts of
elements of a $k_{F}$-basis of  $k_L$.
Let us abbreviate $\O = \O_{KL}$. 

Before proving (iii) we first need to establish
\begin{equation}
  \label{eq:basisclaim}
  \O = \sum_{l=0}^{e-1}\sum_{m=1}^f \pi^l x_m \O_{F}.
\end{equation}
This equality follows by the argument given in the proof
of Proposition II.6.8 \cite{Neukirch}.




We  use $w$ to denote the unique extension of the surjective valuation
$F\rightarrow\IZ\cup\{+\infty\}$
 to a surjective valuation $KL\rightarrow e^{-1}\IZ\cup\{+\infty\}$.

Suppose $\sigma\in \gal{KL/L}\cap\ram{KL/F}{i}$. Then 
$e w(\sigma(a)-a)\ge i+1$ for all $a\in \O$ because $KL/F$ has
ramification index $e$. Because $K/F$ has the same ramification index
we get $\sigma|_K\in \ram{K/F}{i}$. This shows the first claim in part
(iii).

The homomorphism  in (iii) is injective by part (i). It
remains to show that any 
$\sigma' \in \ram{K/F}{i}$ lies in its image. By (i) we can find a
unique lift $\sigma \in \gal{KL/L}$ with $\sigma|_K = \sigma'$. 
It now suffices to show
 $\sigma\in \ram{KL/F}{i}$. 

Suppose $a\in \O$. By (\ref{eq:basisclaim})
we may write
$a = \sum_{l,m}  \pi^l x_m a_{lm}$ for some $a_{lm}\in \O_{F}$.
We have $\sigma(a_{lm}) = a_{lm}$ and $\sigma(x_m)=x_m$ because these
elements lie in $L$. 
We remark $e w(\sigma(\pi^l)-\pi^l) = e w(\sigma'(\pi^l)-\pi^l) \ge
i+1$ since $\pi \in \O_K$.
The ultrametric triangle inequality  gives
\begin{alignat*}1
  ew(\sigma(a) - a) &= ew\left(\sum_{l,m}\sigma(\pi^l x_m a_{lm}) - 
\pi^l x_m a_{lm}\right)
= ew\left(\sum_{l,m}(\sigma(\pi^l) - \pi^l)x_m a_{lm}\right) \\
&\ge \min_{l,m} ew((\sigma(\pi^l)-\pi^l)x_m a_{lm}) \ge i+1.
\end{alignat*}
This yields $\sigma \in \ram{KL/F}{i}$, as desired.
\end{proof}


\section{Supersingular Reduction and Lubin-Tate Theory}
\label{sec:supersingular}

Let $E$ be any elliptic curve defined over a field $K$.
If $N\in\IN$ then 
 $[N]$ stands for the multiplication-by-$N$ endomorphism of
$E$.
The group of points of $E$ of order dividing $N$ that 
are defined over an algebraic closure of $K$ is denoted with $E[N]$.
 If $\ell$ is a prime,
the $\ell$-adic Tate module $\T{E}{\ell}$ of $E$ is the inverse limit over $E[\ell^n]$ as
$n$ runs over the positive integers.
If the 
characteristic of the base field is  different from $\ell$
then $\T{E}{\ell}$ 
 is a torsion free $\IZ_\ell$-module of rank $2$.

Throughout this section we work with the following objects.
Let $p$ be a prime number with $p \ge 5$ and set $q=p^2$.
Suppose $E$ is an elliptic curve defined  over $\IQ_{q}$ 
presented  by a minimal short Weierstrass equation
$y^2 = x^3 + ax + b$ with $a,b\in \IZ_{q}$.
We assume that $E$ has good 
 supersingular
reduction $\widetilde E$. We remark that $\widetilde E$ is
 an elliptic curve defined over $\IF_{q}$. 
For technical reasons we shall suppose that $\widetilde j\in\IF_q$, the $j$-invariant of
$\widetilde E$, is not among $0$ or $1728$. 

Say $f\in\IN$. Because $E$ is fixed we ease notation and
\begin{equation}
\label{eq:easenotation1}
\text{use the symbol $\IQ_{p^f}(N)$ to denote the subfield 
$\IQ_{p^f}(E[N])$ of $\IQpbar$}  
\end{equation}
 generated by the subgroup 
points of $E(\IQpbar)$ whose order divides $N$. 

We begin this section by studying Galois theoretic properties of
torsion points of $E$ of order $M\in \IN$ coprime to $p$.
The first lemma is a basic result from the theory of elliptic curves
of local fields. It makes no use of the fact that $E$ has
supersingular reduction. 
\begin{lemma}
  \label{lem:unramified}
The extension  $\IQ_q(M)/\IQ_q$ is unramified. 
\end{lemma}
\begin{proof}
See Chapter VII \cite{Silverman:AEC}.  
\end{proof}

The previous lemma can be rephrased by stating
$\IQ_q(M)\subset \IQpunr$. 
Let $\ell$ be a prime with $\ell\not=p$. 
The  Galois group  $\gal{\IQpunr/\IQ_q}$ acts on the group of torsion
points of $E$ whose order is a power of $\ell$. 
We obtain a representation
\begin{equation*}
  \rho_\ell : \gal{\IQpunr / \IQ_{q}} \rightarrow \autc{\IZ_\ell}{\T{E}{\ell}}.
\end{equation*}

Reducing modulo $p$ induces an injective
$\IZ_\ell$-module homomorphism
$\T{E}{\ell} \rightarrow
\T{\widetilde E}{\ell}$, cf. Chapter VII \cite{Silverman:AEC}.
After extending scalars this yields
an isomorphism
\begin{equation*}
  \T{E}{\ell}\otimes_{\IZ_\ell}\IQ_\ell \rightarrow \T{\widetilde
E}{\ell}\otimes_{\IZ_\ell}\IQ_\ell
\end{equation*}
of $\IQ_\ell$-vector spaces.

Recall that $\varphi_q \in \gal{\IQpunr/\IQ_q}$ is the lift of 
Frobenius squared. We let $\widetilde \varphi_q$ denote the $q$-Frobenius
endomorphism of $\widetilde E$. 
Then the characteristic polynomial 
 of $\rho_\ell(\varphi_q)$ considered as an automorphism of $\T{E}{\ell}$
  equals the characteristic polynomial
 of the action of $\widetilde\varphi_q$ on 
 $\T{\widetilde E}{\ell}$.
So the determinant of $\rho_\ell(\varphi_q)$  is the  degree of
$\widetilde\varphi_q$ and hence equal to $q$. 
By the Weil Conjectures for elliptic curves defined over finite fields, 
the trace of $\rho_\ell(\varphi_q)$ is an integer $a_q$  which does not
depend on $\ell$. It
  satisfies $|a_q|\le 2 \sqrt{q}=2p$ by Hasse's Theorem. 

In the next lemma  we  use 
  supersingularity for the first time.

\begin{lemma}
\label{lem:frobscalar}
We have  $a_q=\pm 2p$.
Moreover, if $\ell$ is a prime with $\ell\not=p$
then $\widetilde{\varphi_q} = [a_q/2]$ and
 $\rho_\ell(\varphi_q)=a_q/2$.
\end{lemma}
\begin{proof}
Because $\widetilde E$ is assumed to be supersingular we have
$p|a_{q}$. We give a short proof of this well-known fact.
Theorem 13.6.3
\cite{Husemoeller} implies
$\widetilde{\varphi}_{q}^{m} = [p^{m'}]$ on $\widetilde E$ 
  for certain positive integers $m$ and $m'$.
The degree  of $\widetilde{\varphi}_q$ is $q=p^2$ and that of $[p]$
is also $p^2$. Hence $m=m'$ and  
  $\lambda_1^m =\lambda_2^m= p^m$ where $\lambda_{1,2}$ are the eigenvalues of
the action of $\widetilde{\varphi}_q$ on $\T{\widetilde E}{\ell}$. 
Therefore, $\lambda_{1,2}/p$ are
algebraic integers. But $a_{q}/p = (\lambda_1+\lambda_2)/p$ is rational,
so $p | a_q$.

We have already seen  $|a_q|\le 2p$. So 
we may write $a_q=\epsilon p$ with 
 $\epsilon \in \{0,\pm 1,\pm 2\}$.
To show the first claim we will  need to 
 eliminate the cases $\epsilon=0,\pm 1$.

The Theorem of Cayley-Hamilton implies
that $\widetilde{\varphi}_q^2 - [a_q]\circ \widetilde{\varphi}_q+[q]$,
taken as an endomorphism of  $\T{\widetilde E}{\ell}$, vanishes.
 Hence
as an endomorphism of $\widetilde E$ we have
\begin{equation}
  \label{eq:CHthm}
\widetilde{\varphi}_q^2 - [a_q]\circ \widetilde{\varphi}_q+[q]=0.
\end{equation}

Suppose we have $|\epsilon|\le 1$. Since $[p]:\widetilde E\rightarrow
\widetilde E$ is purely inseparable of degree $q$ it follows that
$u\circ [p]=\widetilde{\varphi}_q$ with $u$ an automorphism of $\widetilde E$,
cf. Proposition 13.5.4 \cite{Husemoeller}. 
Now  $\widetilde{\varphi}_q^{2} - [a_{q}]\circ \widetilde{\varphi}_q +  [q]=0$ implies 
$u^2 - [\epsilon]\circ u + 1=0$. 
If for example $\epsilon=0$, then $u$ is an automorphism of order
$4$.
This is incompatible with $\widetilde j\not=1728$ by Theorem III.10.1
  \cite{Silverman:AEC}.
If $\epsilon = \pm 1$ then $u$ has order $6$ or $3$.
On consulting the same reference we arrive at a contradiction because
 $\widetilde j\not = 0$.

Hence $a_q =\pm 2p$ and the first claim holds.

We may thus  rewrite (\ref{eq:CHthm}) as 
 $(\widetilde{\varphi}_q - [a_q/2])^2=0$. The
endomorphism ring of $\widetilde E$ has no zero divisors, so
 $\widetilde{\varphi}_q = [a_q/2]$.
This implies $\rho_\ell(\varphi_q)=a_q/2$ since the reduction
homomorphism is injective.
\end{proof}


We come to the Galois theoretic analysis of torsion points on $E$ 
with order a power of $p$.
Our main tool is the theory of Lubin-Tate modules and its relation
to 
local class field theory. 

\begin{lemma}
\label{lem:LT}
Say $n\in\IN$. 
\begin{enumerate}
\item [(i)] The extension $\IQ_q(p^n) / \IQ_{q}$ is  totally ramified
and abelian
of  degree $(q-1) q^{n-1}$.
Moreover, 
\begin{equation}
\label{eq:Galois2}
 \gal{\IQ_q(p^n)/\IQ_q(p^{n-1})}\cong
 (\IZ/p\IZ)^2\quad\text{if}\quad n\ge 2
\end{equation}
and
\begin{equation}
\label{eq:Galois1}
 \gal{\IQ_q(p^n)/\IQ_q}\cong \IZ/(q-1)\IZ\times(\IZ/p^{n-1}\IZ)^2.
\end{equation}
\item[(ii)] Let $k$ and $i$ be integers with $1\le
  k\le n$ and $q^{k-1} \le i \le q^{k}-1$. The
higher ramification groups are given by
  \begin{equation*}
    \ram{\IQ_q(p^n)/\IQ_q}{i}= \gal{\IQ_q(p^n)/\IQ_q(p^k)}.
  \end{equation*}
\item[(iii)] 
Recall that $M\in\IN$ is coprime to $p$.
The image of the representation $\gal{\IQ_q(p^n)/\IQ_q}\rightarrow
\aut{E[p^n]}$
contains multiplication by $\pm M$ and acts transitively on torsion points
of order $p^n$. 
\end{enumerate}
\end{lemma}
\begin{proof}
We use $a_q=\pm 2p$ from  Lemma \ref{lem:frobscalar}.

Let us first prove the current lemma if $a_{q}=2p$. 
In this case we have
\begin{equation}
\label{eq:multbyp}
 \widetilde{\varphi_q} =  [p]  \quad\text{on}\quad \widetilde{E}. 
\end{equation}

Taking $-x/y$ as a local parameter at
the origin of $E$ determines the
 formal group law associated to $E$, cf. Chapter IV \cite{Silverman:AEC}. We let $[p](T) \in
\IZ_{q}\llbracket T\rrbracket$  denote 
the multiplication-by-$p$ power series, then 
\begin{equation}
\label{eq:pTmodT2}
  [p](T) \equiv pT \mod T^2\IZ_{q}\llbracket
T\rrbracket.
\end{equation}

The reduction of $[p](T)$ modulo $p$ is the multiplication-by-$p$
power series of the formal group associated to $\widetilde E$. 
Relation
(\ref{eq:multbyp}) implies
\begin{equation*}
  [p](T) \equiv T^{q} \mod p\IZ_{q}\llbracket T\rrbracket.
\end{equation*}

This congruence and (\ref{eq:pTmodT2}) imply that
 $[p](T)$ is a Lubin-Tate series, cf. Chapter V \S 2 and \S 4
\cite{Neukirch}. 
It follows from the theory  as laid out in 
\loccitd that the formal group associated to $E$ is a Lubin-Tate
module over $\IZ_p$ for the prime element $p$.

Since $E$ has supersingular reduction, its reduction  has no
torsion points of order divisible by $p$.
By Proposition VII.2.2
  \cite{Silverman:AEC}
the group of $p^n$-division points of said
Lubin-Tate module is isomorphic to $E[p^n]$. 
We will identify both groups since said isomorphism is
compatible with the action of $\gal{\IQpbar/\IQ_q}$. 

Theorem V.5.4 \cite{Neukirch} implies 
that $\IQ_{q}(p^n)/\IQ_{q}$ is totally ramified
and of degree $(q-1)q^{n-1}$. 
The same result stipulates that
$\gal{\IQ_q(p^n)/\IQ_q}$ is isomorphic to $\IZ_q^\times /
\IZ_q^{(n)}$
with $\IZ_q^{(n)}$  the
$n$-th higher unit group of $\IZ_q$.  
Let us consider the short exact sequence
\begin{equation*}
  1 \rightarrow \IZ_q^{(1)}/\IZ_q^{(n)} \rightarrow
\IZ_q^\times / \IZ_q^{(n)}\rightarrow \IZ_q^\times /
\IZ_q^{(1)}\rightarrow 1.
\end{equation*}
The group $\IZ_q^{(1)}/\IZ_q^{(n)}$ is isomorphic to 
$p\IZ_q/p^n\IZ_q \cong (\IZ/p^{n-1}\IZ)^2$
by Proposition II.5.5
 \cite{Neukirch}. 
On the other hand 
 $\IZ_q^{\times}/\IZ_q^{(1)}$ is cyclic of order $q-1$ by Proposition 
II.3.10 \loccitd 
The exact sequence above splits since the groups on the outside have
coprime orders. We conclude (\ref{eq:Galois1}).

The Galois group in (\ref{eq:Galois2}) is the kernel of the surjective
homomorphism
$\gal{\IQ_q(p^n)/\IQ_q}\rightarrow \gal{\IQ_q(p^{n-1})/\IQ_q}$.
Statement (\ref{eq:Galois2}) now follows  from (\ref{eq:Galois1}) and
elementary group theory. This concludes the proof of part (i) when 
$a_q = 2p$.

The statement on the higher ramification
groups in part (ii) is Proposition V.6.1 \cite{Neukirch}.

We now come to part (iii).
The first claim follows
 from Theorem  V.5.4 \cite{Neukirch}.
 Indeed
 we have identified $E[p^n]$ with the $p^n$-torsion
points of the Lubin-Tate module introduced above. 
We will obtain
a field automorphism inducing multiplication by $M$ on $E[p^n]$ by using 
the local norm residue symbol from local class field theory
\begin{equation*}
  (\,\cdot\,,\IQ_q(p^n)/\IQ_q) : \IQ_q(p^n)^\times \rightarrow 
\gal{\IQ_q(p^n)/\IQ_q}.
\end{equation*}
The Theorem of Lubin and Tate, see  V.5.5 \cite{Neukirch},
states that
$(\pm M^{-1},\IQ_q(p^n)/\IQ_q)$ acts on 
$E[p^n]$ as multiplication by $\pm M$.
For the second claim we need in addition
 Proposition V.5.2 {\it ibid.}

The proof of the lemma is complete in the case $a_{q}=2p$.
We shall not neglect the case  $a_{q}=-2p$
since this  occurs  if $a$ and $b$ happen to lie in $\IZ_{p}$, cf. the
example following this proof. 
We will reduce to the case already proved by 
twisting $E$.
This has the effect of flipping the sign of $a_q$.
The details are as follows.

Because $p\not = 2$ there exists $t\in \IZ_{q}$ 
which is not a square modulo $p$.
In particular, $t\not\in p\IZ_q$ and $\IQ_q(t^{1/2})/\IQ_q$
is an unramified
quadratic extension. In other words $\IQ_q(t^{1/2})=\IQ_{q^2}$.

Let us consider the quadratic twist
$E_t$ of $E$ determined by $y^2 = x^3+at^2x + bt^3$. It too has good
reduction $\widetilde E_t$ which  is a quadratic twist of $\widetilde
E$. 
We note that 
 $\widetilde E_t(\IF_{q}) = q+ 1 - a'_{q}$
with $a'_{q}$ the
 trace of the $q$-Frobenius of $\widetilde E_{t}$.
By Proposition 13.1.10 \cite{Husemoeller} we find
$a'_{q}=-a_{q}=2p$. So we may apply the current lemma to $E_t$. 

The elliptic curves $E$ and $E_t$ are isomorphic over
$\IQ_{q^2}$.
Indeed, $(x,y)\mapsto (tx,t^{3/2}y)$ determines an isomorphism
$\chi:E\rightarrow E_t$. 
Hence
\begin{equation}
\label{eq:fieldequality}
  \IQ_{q^2}(E_t[p^n]) = \IQ_{q^2}(p^n).
\end{equation}

We claim that $\IQ_q(p^n)/\IQ_q$ is totally ramified. 
Recall that $\IQ_q(E_t[p^n])/\IQ_q$ is totally
ramified.  Lemma \ref{lem:lift}(i) and (\ref{eq:fieldequality}) imply that the inertia degree of 
$\IQ_{q^2}(p^n)/\IQ_q$ is $2$. In order to prove our claim it
suffices to show that the unramified extension
 $\IQ_{q^2}(p^n)/\IQ_q(p^n)$  is non-trivial. For then it 
is of degree $2$ 
and must account for the full residue field extension
of $\IQ_{q^2}(p^n)/\IQ_q$. 
Restriction induces an isomorphism
between the  groups $\gal{\IQ_{q^2}(E_t[p^n])/\IQ_q}$ and
 $\gal{\IQ_q(E_t[p^n])/\IQ_q}\times \gal{\IQ_{q^2}/\IQ_q}$. 
So there is $\sigma\in \gal{\IQ_{q^2}(E_t[p^n])/\IQ_q}$
with $\sigma(t^{1/2})=-t^{1/2}$.
In view of statement (iii) of this lemma applied to the elliptic curve
$E_t$ we may arrange that $\sigma$ acts on $E_t[p^n]$ as $[-1]$.
Suppose $S=(x,y)\in E[p^n]$. Using
$\chi(S)\in E_t[p^n]$ we find
\begin{equation*}
 [-1](\chi(S)) = \sigma(\chi(S)) = (\sigma(tx),\sigma(t^{3/2}y))
 = (t\sigma(x),-t^{3/2}\sigma(y)) = [-1](\chi(\sigma(S)))
\end{equation*}
which implies $S=\sigma(S)$. So $\sigma$ fixes the field
$\IQ_q(p^n)$. We conclude $\IQ_q(p^n) \not= \IQ_{q^2}(p^n)$
because $\sigma$ is not trivial. Our claim from above follows
and with it the first assertion of (i) for $E$.

By Lemma \ref{lem:lift}(i)  restriction induces  isomorphisms
$\gal{\IQ_{q^2}(p^n)/\IQ_{q^2}(p^k)}\rightarrow
\gal{\IQ_q(p^n)/\IQ_q(p^k)}$
and 
$\gal{\IQ_{q^2}(E_t[p^n])/\IQ_{q^2}(E_t[p^k])}\rightarrow
\gal{\IQ_q(E_t[p^n])/\IQ_q(E_t[p^k])}$ of groups
for $0\le k\le n$. 
So
\begin{equation*}
 \gal{\IQ_q(p^n)/\IQ_q(p^k)}\cong
 \gal{\IQ_q(E_t[p^n])/\IQ_q(E_t[p^k])} 
\end{equation*}
implies the remaining assertions of
part (i). 

Let us  prove (iii) before (ii). By what has already been shown, there
is $\sigma \in
\gal{\IQ_q(E_t[p^n])/\IQ_q}$ that acts on $E_t[p^n]$ as multiplication by
$\pm M$.
We may lift $\sigma$ uniquely
to $\widetilde \sigma \in \gal{\IQ_{q^2}(p^n)/\IQ_{q^2}}$.
If $S \in E[p^n]$, then $\chi(S)\in E_t[p^n]$.
Because
$\widetilde\sigma$ commutes with $\chi$ we find 
that $\widetilde\sigma$ acts on $S$ as multiplication by $\pm M$. 
The first claim in 
part (iii) follows in general because $S$ was arbitrary. 
The second claim is proved along similar lines. 

Finally, we prove (ii) for $E$.
Say $i\ge -1$.
We now apply Lemma \ref{lem:lift}(iii) to
 the unramified extension 
$\IQ_{q^2}/\IQ_q$ and both
 totally ramified extensions $\IQ_{q}(p^n)/\IQ_q$
and $\IQ_q(E_t[p^n])/\IQ_q$.
 We find  isomorphisms of groups

\begin{equation*}
\if \usexypic 1
 \xymatrix  @C=-90pt @R=30pt 
{
 & \ar[ld] 
\gal{\IQ_{q^2}(p^n)/\IQ_{q^2}}\cap
   \ram{\IQ_{q^2}(p^n)/\IQ_{q}}{i} 
\ar[rd] &  \\
 \ram{\IQ_{q}(p^n)/\IQ_q}{i} & &   \ram{\IQ_{q}(E_t[p^n])/\IQ_q}{i}
}  
\else
\text{(Some diagram)}
\fi
\end{equation*}
 which are induced by
restrictions.
Part (ii)  follows formally  from this diagram and 
since $\chi$ is defined over $\IQ_{q^2}$.
\end{proof}

Twisting is necessary  to obtain a Lubin-Tate
series. To see why let us consider for the moment the case $p=5$
and elliptic curve  defined by $y^2=x^3 + 5x + 1$. 
It has good supersingular reduction with $a_{25} = -10$. The  multiplication-by-$5$
power series of the associated formal group satisfies 
\begin{equation*}
  [5](T) \equiv -T^{25}  \mod 5 \IZ_{25}\llbracket T\rrbracket.
\end{equation*}
 It is not a Lubin-Tate series because of the wrong sign.
However,  twisting by $\sqrt{2} \in \IZ_{25}$
gives the Weierstrass equation $y^2=x^3 + 10x + 2\sqrt 2$ which 
leads to
\begin{equation*}
  [5](T) \equiv T^{25} \mod 5 \IZ_{25}\llbracket T\rrbracket.
\end{equation*}

Recall that 
 $M\in\IN $ is  coprime to $p$ and suppose $n$ is a non-negative
integer. We set $N=p^nM$. 

Now we collect useful Galois theoretic statements 
involving the extension $\IQ_q(N)/\IQ_q$.

\begin{lemma}
  \label{lem:galoisprops}
The following statements hold.
\begin{enumerate}
\item[(i)]  The composition $\IQ_q(p^n) \IQ_q(M)$ is
  $\IQ_q(N)$. 
\item [(ii)]
The extension $\IQ_q(N)/\IQ_q(p^n)$ is unramified
and the extension $\IQ_q(N)/\IQ_q(M)$ is totally ramified.
\item[(iii)] Restricting to $\IQ_q(p^n)$ induces an isomorphism of
  groups
  \begin{equation*}
    \gal{\IQ_q(N)/\IQ_q(M)}\rightarrow \gal{\IQ_q(p^n)/\IQ_q}.
  \end{equation*}
In particular, $\IQ_q(N)/\IQ_q(M)$ is abelian. 
\item[(iv)]
If $n\ge 1$, then
\begin{equation}
\label{eq:Galois3}
  \gal{ \IQ_q(N)/ \IQ_q(N/p)}\cong \left\{
\begin{array}{ll}
  (\IZ/p\IZ)^2 &: \text{if }n\ge 2,\\
  \IZ/(q-1)\IZ &: \text{if }n= 1.\\
\end{array}
\right.
\end{equation}
\end{enumerate}
\end{lemma}
\begin{proof}
Part (i) follows since any element of $E[N]$ is the sum of an
element in $E[p^n]$ and an element in $E[M]$.

By Lemma \ref{lem:unramified} the extension $\IQ_q(M)/\IQ_q$ is unramified
and Lemma \ref{lem:LT}(i) implies that
$\IQ_q(p^n)/\IQ_q$ is totally ramified. 
Part (ii) now follows from part (i) and Lemma \ref{lem:lift}(ii).

The first statement  in part (iii) follows from part (ii) and 
 Lemma \ref{lem:lift}(i).
The claim on commutativity is then a consequence of Lemma \ref{lem:LT}(i). 

To prove (iv) we first note $\IQ_q(N)  =
\IQ_q(p^n)\IQ_q(N/p)$ by part (i). 
We have a diagram of field extensions
\begin{equation*}
\if \usexypic 1
 \xymatrix  @C=0pt @R=20pt 
{
 & \ar@{-}[ld] 
\IQ_q(N) 
\ar@{-}[rd] &  \\
\ar@{-}[rd]_{\atopx{\text{totally}}{\text{ramified}}} \IQ_q(p^{n})  & & \IQ_q(N/p)
\ar@{-}[ld]^{\atopx{\text{unramified}}{}}  \\
& \IQ_q(p^{n-1}) & 
}
\else
\text{(Some diagram)}
\fi
\end{equation*}
the two claims on ramification  behavoir
 follow from part (ii).
By Lemma \ref{lem:lift}(i),  restricting field automorphisms induces
an isomorphism
\begin{equation*}
  \gal{ \IQ_q(N)/ \IQ_q(N/p)} \cong
  \gal{\IQ_q(p^n)/\IQ_q(p^{n-1})}.
\end{equation*}
With this  isomorphism (\ref{eq:Galois3}) follows
from (\ref{eq:Galois2}) and (\ref{eq:Galois1}), respectively.
\end{proof}

We state two auxiliary lemmas which are used in later sections. 
The first lemma describes the roots of unity 
in $\IQ_q(N)$ having order a
power of $p$. 

\begin{lemma}
\label{lem:rootofone}
We have $\IQ_q(N)\cap \mu_{p^\infty} = \mu_{p^n}$.
\end{lemma}
\begin{proof}
Properties of the Weil pairing imply
the inclusion ``$\supset$'' from the assertion.

To show the other inclusion we first verify
\begin{equation}
  \label{eq:claim1}
\IQ_q(p^n) \cap \mu_{p^\infty}\subset \mu_{p^n}.
\end{equation}
So let $\zeta$ lie $\IQ_q(p^n)$ and suppose it has order $p^{n'}$. We may assume
 $n'\ge n$. 

If $n=0$, then $\zeta\in\IQ_q$. But $\IQ_p(\zeta)/\IQ_p$ is totally
ramified by Proposition II.7.13 \cite{Neukirch} and is only
trivial if $n'=0$. 
Moreover,  this extension has degree
$[\IQ_q(\zeta):\IQ_q]$ by Lemma \ref{lem:lift}(ii).
So we must have $n'=0$. This proves (\ref{eq:claim1}) if $n=0$. 

We now suppose $n'\ge n\ge 1$.
Restriction induces a surjective homomorphism
$\gal{\IQ_q(p^n)/\IQ_q}\rightarrow \gal{\IQ_q(\zeta)/\IQ_q}$. 
The structure of both Galois groups is known. Indeed, by Lemma
\ref{lem:LT}(i) the group 
on the left
is isomorphic to 
$\IZ/(q-1)\IZ\times (\IZ/p^{n-1}\IZ)^2$.
On the other hand,
 $\gal{\IQ_q(\zeta)/\IQ_q}\cong \gal{\IQ_p(\zeta)/\IQ_p}$
as above by Proposition II.7.13.
The same result also implies
$\gal{\IQ_q(\zeta)/\IQ_q}\cong
(\IZ/p^{n'}\IZ)^\times\cong\IZ/(p-1)\IZ\times \IZ/p^{n'-1}\IZ$, the second
isomorphism holds
 since $p\not=2$. 
A group homomorphism
\begin{equation*}
  \IZ/(q-1)\IZ\times (\IZ/p^{n-1}\IZ)^2\rightarrow
\IZ/(p-1)\IZ\times\IZ/p^{n'-1}\IZ
\end{equation*}
 cannot be surjective if $n'>n$. So $n'=n$.
This shows that $\zeta$ has order $p^n$ and  claim
(\ref{eq:claim1}) holds. 

Now suppose that $\zeta\in \IQ_q(N)$ has order
$p^{n'}$. Again, we may assume $n' \ge n$.
The extension $\IQ_q(N)/\IQ_q(p^n)$ is unramified by Lemma
\ref{lem:galoisprops}(ii), so $\IQ_q(p^n)(\zeta)/\IQ_q(p^n)$ is unramified. 
Using the Weil pairing, we have already proved that 
 $\zeta$ is contained in $\IQ_q(p^{n'}) \supset \IQ_q(p^n)$.
 The extension $\IQ_q(p^{n'})/\IQ_q(p^n)$ is totally
ramified and therefore so is  $\IQ_q(p^n)(\zeta)/\IQ_q(p^n)$. We conclude $\zeta\in \IQ_q(p^n)$
and  thus the lemma follows from   (\ref{eq:claim1}).
\end{proof}

The second lemma will play a role in a descent argument used in a later
section.


\begin{lemma}
\label{lem:qpower}
Let us suppose $n\ge 1$. 
If  $\psi\in\gal{\IQ_q(N)/\IQ_q(N/p)}$
and $\alpha\in \IQ_q(N) \ssm\{0\}$
such that ${\psi(\alpha)}/{\alpha}\in \mu_{\infty}$, then
\begin{equation}
\label{eq:defineQn}
  \frac{\psi(\alpha)}{\alpha} \in \mu_{Q(n)}
\quad\text{where}\quad
  Q(n) =  \left\{
  \begin{array}{cl}
    q &: \text{if } n\ge 2, \\
    (q-1)q &: \text{if }n=1.
  \end{array}
\right.
\end{equation}
\end{lemma}
\begin{proof}
In the following it is useful to write $x^\psi$ for  $\psi(x)$
if $x\in \IQ_q(N)$. 

Let $N'$ denote the order of the root of unity
$\beta = \alpha^{\psi-1} = \psi(\alpha)/\alpha$. We decompose
$N'=p^{n'}M'$ with $n'\ge 0$ and $p\nmid M'$.
The root of unity
 $\xi=\beta^{p^{n'}}\in \IQ_q(N)$ has order $M'$.
The extension $\IQ_q(N)/\IQ_q(M)$ is totally ramified by Lemma
\ref{lem:galoisprops}(ii). This property is shared by the subextension
 $\IQ_q(M)(\xi) / \IQ_q(M)$.
The order of $\xi$ is prime to $p$, so
 $\IQ_p(\xi)/\IQ_p$ is unramified by Proposition II.7.12 \cite{Neukirch}.
Hence $\IQ_q(M)(\xi)/ \IQ_q(M)$ is
unramified.
We find $\xi \in \IQ_q(M)$.
In particular, $\xi$ is fixed by $\psi$. 

The order of $\beta^{M'}$ is $p^{n'}$. Hence 
$n'\le n$ by Lemma \ref{lem:rootofone} and because $\beta \in
\IQ_q(N)$. The same lemma also yields $\beta^{pM'} \in \IQ_q(N/p)$
and so $\psi$ fixes $\beta^{pM'}$.  

Let us fix integer $a$ and $b$ with $1=ap^{n'}+bM'$. Then
$\beta = \xi^a \beta^{bM'}$ and 
so $\psi$ fixes $\beta^p$ since it fixes $\xi^{ap}$ and $\beta^{bpM'}$.

Let $t$ denote the order of $\psi$ as an element of
$\gal{\IQ_q(N)/\IQ_q(N/p)}$. 
Then 
\begin{equation}
\label{eq:ttrick}
1 = \alpha^{p(\psi^t-1)} = \alpha^{p(\psi-1)(\psi^{t-1}+\cdots + \psi+1)}
 = \beta^{p(\psi^{t-1}+\cdots +\psi + 1)} = \beta^{pt}
\end{equation}
because $\beta^{p\psi}=\beta^p$. 
But by Lemma \ref{lem:galoisprops}(iv) the order $t$ is a
divisor of $p$
if $n\ge 2$ and a divisor of $q-1$ if $n=1$. 
The lemma follows from  $pt | Q(n)$.
\end{proof}

This  proof shows that $Q(1)=(q-1)q$ can be replaced by
 the smaller value $(q-1)p$. But it is convenient to have
 $q|Q(n)$ for all $n$ in the proof of Lemma \ref{lem:ptorsion} below.

\section{Local Metric Estimates}
\label{sec:localmetric}

In this section $E$ and $p$ are as in the previous one. Moreover,
$q=p^2$ and $N$ is a positive integer. 
The simplification in notation (\ref{eq:easenotation1}) is  used in this section too.
We recall that $\varphi_q \in \gal{\IQpunr/\IQ_q}$ is the lift of 
Frobenius squared. 

We come to a first metric estimate which is used in the unramified case $p\nmid N$.

\begin{lemma}
\label{lem:metric2}
Suppose $p\nmid N$ and  $\alpha\in \IQ_q(N)$.  Then
$\alpha\in \IQpunr$ and
 \begin{equation}
\label{eq:metric2}
 |\varphi_q(\alpha)-\alpha^q|_p  \le p^{-1}
\max\{1,|\varphi_q(\alpha)|_p\}\max\{1,|\alpha|_p\}^q.
 \end{equation}
\end{lemma}
\begin{proof}
The field $L=\IQ_q(N)$ is an unramified 
 extension of $\IQ_q$
by Lemma \ref{lem:unramified}. This is the first claim. 
To prove the second claim we first assume that
 $\alpha$ is an integer in $L$,
i.e. $|\alpha|_p\le 1$.
Then
 $\varphi_q(\alpha)-\alpha^q$ is in the maximal ideal of
$\O_{L}$. This
 maximal ideal is $p\O_{L}$ since $L/\IQ_q$ is unramified. 
Therefore, $|\varphi_q(\alpha)-\alpha^q|_p \le |p|_p = p^{-1}$
and thus (\ref{eq:metric2}) holds true.

If $\alpha$ is not an integer in $L$, then $\alpha^{-1}$ is and
we have
 $|\varphi_q(\alpha^{-1})-\alpha^{-q}|_p\le p^{-1}$ by what has already
been proved.
The ultrametric triangle inequality  yields
\begin{alignat*}1
|\alpha^{-q}(\varphi_q(\alpha)-\alpha^q)|_p = 
|(\alpha^{-q}-\varphi_q(\alpha^{-1}))\varphi_q(\alpha)|_p
\le p^{-1} |\varphi_q(\alpha)|_p
\end{alignat*}
and our lemma now follows quickly.
\end{proof}

The second metric estimate 
finds application in the ramified case $p|N$.

\begin{lemma}
\label{lem:metric1}
Suppose $p| N$ and  $\alpha\in \IQ_q(N)$. Then
 \begin{equation}
\label{eq:metric1}
 |\psi(\alpha)^{q}-\alpha^{q}|_p 
\le p^{-1} \max\{1,|\psi(\alpha)|_p\}^{q}\max\{1,|\alpha|_p\}^{q}
 \end{equation}
for all $\psi \in \gal{ \IQ_q(N)/ \IQ_q(N/p)}$.
\end{lemma}
\begin{proof}
For brevity we write $K = \IQ_q(p^n)$
and $L=\IQ_q(N/p^n)$ where $n\ge 1$ is the greatest integer with
$p^n|N$. Then $KL = \IQ_q(N)$ by Lemma \ref{lem:galoisprops}(i). 

As in the proof of Lemma \ref{lem:metric2}, we
 first suppose that $\alpha$ is an integer in $\IQ_q(N)$.
We have $\psi|_{K} \in
\gal{K/\IQ_q(p^{n-1})}$.
By Lemma \ref{lem:LT}(ii)  this restriction is in
 $\ram{K/\IQ_{q}}{i}$ with $i = q^{n-1}-1$.
Lemma \ref{lem:galoisprops}(ii) implies that
$\IQ_q(N)/K$ is unramified. 
Now $\psi$ is the unique lift of $\psi|_K$ to
$\IQ_q(N)$ that restricts to the identity on $L$.
By Lemma \ref{lem:lift}(iii) $\psi$ must lie in 
$\ram{\IQ_q(N)/\IQ_q}{i}$. In other words
\begin{equation*}
 \psi(\alpha)-\alpha \in \mathfrak{P}^{q^{n-1}}
\end{equation*}
where $\mathfrak{P}$
 is the maximal ideal of the ring of integers of $\IQ_q(N)$.
The ramification index of $\IQ_q(N)/\IQ_q$ is
 $e = (q-1)q^{n-1}$ by Lemmas \ref{lem:LT}(i) and \ref{lem:galoisprops}(ii).
Therefore, $(\psi(\alpha)-\alpha)^{q} \in \mathfrak{P}^{q^n}\subset
 \mathfrak{P}^e$. Since $p\in \mathfrak{P}^e$  we conclude
 \begin{equation*}
   0 \equiv (\psi(\alpha)-\alpha)^{q} \equiv \psi(\alpha^{q})-
\alpha^{q} \mod \mathfrak{P}^e.
 \end{equation*}
This leads to 
  $|\psi(\alpha)^{q}-\alpha^{q}|_p \le |p|_p = p^{-1}$.
Hence (\ref{eq:metric1}) holds true if $\alpha$ is an integer in $\IQ_q(N)$.

Deducing this lemma for non-integral $\alpha$ is done as in the
previous lemma. If
 $\alpha^{-1}$ is an integer in $KL$ then
 $|\psi(\alpha^{-1})^q - \alpha^{-q}|_p \le p^{-1}$.
The ultrametric triangle inequality implies
\begin{equation*}
  |\alpha^{-q}(\psi(\alpha)^{q} - \alpha^{q})|_p =
 |(\alpha^{-q} - \psi(\alpha^{-1})^{q}) \psi(\alpha)^q|_p
\le p^{-1} |\psi(\alpha)|_p^q 
\end{equation*}
and we immediately obtain (\ref{eq:metric1}).
\end{proof}

\section{Globalization and a First Lower Bound}
\label{sec:global}

We cease working over a local field and now suppose 
that $E$ is an elliptic curve defined over $\IQ$. Furthermore, 
 $p \ge 5$ is a fixed prime and $q=p^2$

We introduce two properties associated to   $E$
and $p$.

\begin{enumerate}
\item [(P1)] The elliptic curve
$E$ has good
supersingular reduction at $p$
and the $j$-invariant of this reduction is not among $\{0,1728\}$.
\item[(P2)] The natural Galois representation
  \begin{equation*}
    \gal{\IQbar/\IQ}\rightarrow \aut{E[p]}
  \end{equation*}
is surjective.
\end{enumerate}

Only the first property will play a role in the current section.
If it is satisfied, then  the results stated in Sections
\ref{sec:supersingular} and \ref{sec:localmetric} apply to $E$  
considered as an
elliptic curve   over the field $\IQ_q$.

Say $K$ is a Galois extension of $\IQ$ and let $v$ be a place of $K$. 
An automorphism $\sigma\in \gal{K/\IQ}$ determines an new place
$\sigma v$ of $K$ through
\begin{equation*}
  |\alpha|_{\sigma v} = |\sigma^{-1}(\alpha)|_v
\end{equation*}
for all $\alpha\in K$.

Let $N$ be a positive integer.
 In
addition to the convention 
(\ref{eq:easenotation1}) we also 
\begin{equation}
\label{eq:easenotation2}
\text{use the symbol $\IQ(N)$ to denote 
the field $\IQ(E[N])$}.
\end{equation}
The number field $\IQ(N)$ is a Galois extension of $\IQ$.

We now  get a height lower bound in the unmramified case $p\nmid N$.

\begin{lemma}
\label{lem:primetoptorsion}
Suppose $E$ and $p$ satisfy (P1).
We assume $p\nmid N$.
 If $\alpha\in\IQ(N)\ssm\mu_\infty$ is non-zero,
then
\begin{equation*}
  \height{\alpha}\ge \frac{\log(p/2)}{p^2+1}. 
\end{equation*}
\end{lemma}
\begin{proof}
We recall that all  number fields are taken to be  subfields of
$\IQpbar$. 

Suppose
  $\ell$ is a prime divisor of $N$ and $\ell^m|N$ with $m\in\IN$
but $\ell^{m+1}\nmid N$.  
Then $\ell\not=p$ by hypothesis. Lemma \ref{lem:frobscalar} implies that
 $\varphi_q$ acts on $E[\ell^m]$ as multiplication
by $a_q/2\in\IZ$.

Taking the sum of points leads 
to a isomorphism of a direct sum over all $E[\ell^m]$
with $\ell^m$ as above and $E[N]$.
This isomorphism is compatible with the action of the Galois group.
We deduce that $\varphi_q$ acts on  $E[N]$ as
multiplication by $a_q/2$.
So the restriction $\varphi_q|_{\IQ(N)}$, which we identify with
$\varphi_q$, 
lies in the center of $\gal{\IQ(N)/\IQ}$.


The restriction of $|\cdot|_p$ to $\IQ(N)$ is  a place 
 $v$.

We define $x  = \varphi_q(\alpha)-\alpha^q\in \IQ(N)$ and claim that
$x\not=0$. Indeed, otherwise we would have
$\height{\varphi_q(\alpha)} = \height{\alpha^q}$. 
Conjugating does not affect the height, so
$\height{\alpha}=\height{\alpha^q}=q\height{\alpha}$
and hence
 $\height{\alpha}=0$. Therefore $\alpha=0$ or
$\alpha\in\mu_\infty$ by Kronecker's Theorem. This contradicts our
assumption on $\alpha$.

Since $x\not=0$, the product formula implies
\begin{equation}
\label{eq:PF1}
   \sum_{w} d_w \log |x|_w=0
\end{equation}
where the sum is over all places of $\IQ(N)$.

Say $w$ is a finite place of $\IQ(N)$ above $p$. Then 
$w=\sigma^{-1}v$ for some $\sigma\in \gal{\IQ(N)/\IQ}$ because
the Galois group acts transitively on the places of $\IQ(N)$ above $p$. 
The fact that
$\varphi_q$ and $\sigma$ commute
gives the second equality in 
\begin{equation*}
|x|_w = 
  |\sigma(\varphi_q(\alpha))-\sigma(\alpha)^q|_{v} = 
|\varphi_q(\sigma(\alpha))-\sigma(\alpha)^q|_v.
\end{equation*}
Now we estimate the right-hand side from above using 
Lemma \ref{lem:metric2} applied to
$\sigma(\alpha)$
\begin{alignat}1
\label{eq:unrambound1}
  |x|_w &\le
  p^{-1}\max\{1,|\varphi_q(\sigma(\alpha))|_v\}\max\{1,|\sigma(\alpha)|_v\}^q
  \\
\nonumber
  &=
  p^{-1}\max\{1,|\sigma(\varphi_q(\alpha))|_v\}\max\{1,|\sigma(\alpha)|_v\}^q
  \\
\nonumber
&=  p^{-1}\max\{1,|\varphi_q(\alpha)|_w\}
\max\{1,|\alpha|_w\}^q.
\end{alignat}

If $w$ is an arbitrary finite place of $\IQ(N)$, the
 ultrametric triangle
inequality gives
\begin{equation}
\label{eq:unrambound2}
  |x|_w \le \max\{|\varphi_q(\alpha)|_w,|\alpha^q|_w\}
\le \max\{1,|\varphi_q(\alpha)|_w\}\max\{1,|\alpha|_w\}^q.
\end{equation}

Finally, if $w$ is an infinite place of $\IQ(N)$, the triangle inequality
implies
\begin{equation}
\label{eq:unrambound3}
  |x|_w \le 2\max\{|\varphi_q(\alpha)|_w,|\alpha^q|_w\}
\le 2\max\{1,|\varphi_q(\alpha)|_w\}\max\{1,|\alpha|_w\}^q.
\end{equation}

We apply the logarithm to the bounds
(\ref{eq:unrambound1}), (\ref{eq:unrambound2}), and
(\ref{eq:unrambound3}), take the sum over all places $w$ of $\IQ(N)$
 with 
multiplicities $d_w$, and use the product formula (\ref{eq:PF1}) to find
\begin{alignat*}1
0&=    \sum_{w|p}d_w\log |x|_w + 
\sum_{w\nmid \infty, w\nmid p} d_w
\log |x|_w +  \sum_{w|\infty}d_w\log |x|_w \\
&\le 
-\sum_{w|p}d_w\log p
+\sum_{w|\infty} d_w \log 2 
+\sum_{w} d_w\log(\max\{1,|\varphi_q(\alpha)|_w\}
\max\{1,|\alpha|_w\}^{q}).
\end{alignat*}
We divide this expression by $[\IQ(N):\IQ]$ and use
(\ref{eq:localdegrees})  together with the definition of the
height given in Section \ref{sec:height} to obtain
\begin{equation*}
0 \le 
   -\log p + \log 2+ \height{\varphi_q(\alpha)} + q\height{\alpha}.
\end{equation*}
Hence $\height{\varphi_q(\alpha)}+q\height{\alpha}\ge \log(p/2)$. 
The lemma  follows from $q=p^2$ and $\height{\varphi_q(\alpha)} =
\height{\alpha}$, one of our basic height properties.
\end{proof}

The remainder of the proof of Theorem \ref{thm:main} concerns the study of the more delicate
ramified case,  i.e. when  $p|N$. 
Instead of working with a lift of the Frobenius automorphism, we use
an element in a higher ramification group. 
The next lemma addresses the issue
that ramification groups  need not lie in the center of the global 
Galois group.

\begin{lemma}
\label{lem:largeorbit}
Suppose $E$ and $p$ satisfy (P1). We assume $p|N$. Suppose
$\psi \in \gal{\IQ_q(N)/\IQ_q(N/p)}$ 
which we identify 
with its restriction to $\IQ(N)$. If
\begin{equation*}
  G = \{\sigma\in \gal{\IQ(N)/\IQ};\,\,
  \sigma\psi\sigma^{-1}=\psi\}
\end{equation*}
is its centralizer, then
\begin{equation*}
  \# G v \ge \frac{1}{p^4}\frac{[\IQ(N):\IQ]}{d_v}
\end{equation*}
where $v$ is the place of $\IQ(N)$ induced by $|\cdot |_p$. 
\end{lemma}
\begin{proof}
We define the normal subgroup
\begin{equation*}
 H = 
\gal{\IQ(N)/\IQ(N/p)}.
\end{equation*}
 of $\gal{\IQ(N)/\IQ}$; it contains
 $\psi$. 

We fix an isomorphism between $E[N]\cong
 (\IZ/N\IZ)^2$ allowing us to represent an
 automorphism of $E[N]$  
 by an element
of $\gl{\IZ/N\IZ}{2}$. 
An automorphism of $E[N]$ acting trivially on $E[N/p]$ 
is represented by an element of 
$1+N/p \mat{\IZ/N\IZ}{2}$. Since the representation
$\gal{\IQ(N)/\IQ}\rightarrow
\gl{\IZ/N\IZ}{2}$ is injective, we have
\begin{equation}
\label{eq:Hbound}
  \# H \le p^4.
\end{equation}

The orbit of $\psi$ under the action of conjugation 
by $\gal{\IQ(N)/\IQ}$ is contained in the normal
subgroup $H$.
The stabilizer of $\psi$ under this action is 
the centralizer $G$
 from the assertion. 
So we may bound
\begin{equation}
\label{eq:Gbound}
  \#G \ge \frac{[\IQ(N):\IQ]}{\#H} \ge \frac{[\IQ(N):\IQ]}{p^4}
\end{equation}
using (\ref{eq:Hbound}).

Restricting  $|\cdot|_p$ determines a
place $v$ of $\IQ(N)$ lying 
 above $p$.
The Galois group $\gal{\IQ(N)/\IQ}$ acts transitively on all places of
$\IQ(N)$ lying above $p$ and the
 total number of such places
is
\begin{equation*}
  \frac{[\IQ(N):\IQ]}{d_v} = \frac{[\IQ(N):\IQ]}{[\IQ_p(N):\IQ_p]}.
\end{equation*}
So the orbit $Gv$ of $v$ under the action of the group $G$
has cardinality
\begin{alignat*}1
  \# G v
&\ge \frac{1}{[\gal{\IQ(N)/\IQ}:G]}
\frac{[\IQ(N):\IQ]}{d_v}
\ge
\frac{1}{p^4}
\frac{[\IQ(N):\IQ]}{ d_v}.
\end{alignat*}
by (\ref{eq:Gbound}).
\end{proof}


At first we will only get a weak height inequality which holds for
algebraic numbers satisfying a different condition than in Theorem
\ref{thm:main}. We recall that the expression $Q(n)$ was
defined in (\ref{eq:defineQn}).

\begin{lemma}
\label{lem:ptorsion}
Suppose $E$ and $p$ satisfy (P1). We assume $p|N$ and
let $n\ge 1$ be the greatest
integer with $p^n|N$.
If  $\alpha\in\IQ(N)$
 satisfies 
 $\alpha^{Q(n)} \not\in \IQ_q(N/p)$, there
 exists a non-zero
$\beta\in \IQbar \ssm\mu_\infty$ with 
$\height{\beta}\le 2p^4\height{\alpha}$ and
  \begin{equation}
\label{eq:ptorsionineq}
\height{\alpha} + 
\max\left\{0,\frac{1}{[\IQ(\beta):\IQ]}
\sum_{\tau}
\log |\tau(\beta)-1|\right\}
\ge \frac{\log p}{2p^8}
  \end{equation}
where the sum runs over all field embeddings $\tau : \IQ(\beta)\rightarrow \IC$.
\end{lemma}
\begin{proof}
For brevity, we set $Q=Q(n)$.
By hypothesis we may choose $\psi\in \gal{\IQ_q(N)/ \IQ_q(N/p)}$
with $\psi(\alpha^{Q})\not=\alpha^{Q}$. We note that
$\alpha\not=0$.

We define 
\begin{equation*}
  x = \psi(\alpha^{Q})-\alpha^{Q} \in  \IQ(N)
\end{equation*}
and observe $x\not=0$ by our choice of $\psi$.
So
\begin{equation}
\label{eq:PF2}
    \sum_{w} d_w \log |x|_w =0
\end{equation}
by the product formula.

Say $G$ and $v$ are as in Lemma \ref{lem:largeorbit}. 
Let $\sigma\in G$. The place $\sigma v$  of $\IQ(N)$
satisfies
 $|\sigma(y)|_{\sigma v} = |y|_v$
for all $y\in \IQ(N)$. 
So 
$|(\sigma\psi\sigma^{-1})(\alpha^{Q}) - \alpha^{Q}|_{\sigma v}=
 |\psi(\sigma^{-1}(\alpha)^{Q})-\sigma^{-1}(\alpha)^{Q}|_v$.
By definition we have $q|Q$, so we may apply
  Lemma \ref{lem:metric1} to $\sigma^{-1}(\alpha)^{Q/q}$.
This yields
\begin{alignat*}1
|(\sigma\psi\sigma^{-1})(\alpha^{Q}) - \alpha^{Q}|_{\sigma v}
&\le p^{-1}
\max\{1,|\psi(\sigma^{-1}(\alpha))|_v\}^{Q}\max\{1,|\sigma^{-1}(\alpha)|_v\}^{Q}\\
&\le p^{-1}
\max\{1,|(\sigma\psi\sigma^{-1})(\alpha))|_{\sigma v}\}^{Q}
\max\{1,|\alpha|_{\sigma v}\}^{Q}
\end{alignat*}
Now $\sigma\psi\sigma^{-1}=\psi$ since $\sigma\in G$. Therefore, 
\begin{equation}
\label{eq:sigmainG}
  |x|_{w} \le p^{-1} \max\{1,|\psi(\alpha)|_{w}\}^{Q}
\max\{1,|\alpha|_{w}\}^{Q}
\quad\text{for all}\quad w\in Gv.
\end{equation}

If $w$ is an arbitrary finite place of $\IQ(N)$,  the
ultrametric triangle inequality implies
\begin{equation}
\label{eq:finiteplace}
  |x|_w \le \max\{1,|\psi(\alpha)|_w\}^Q
\max\{1,|\alpha|_w\}^Q.
\end{equation}

Say $w$ is an infinite place. Applying the triangle inequality as
 for example in (\ref{eq:unrambound3}) to bound $|x|_w$ would lead
to a ruinous factor $2$. 
Instead we define
\begin{equation*}
  \beta = \frac{\psi(\alpha^Q)}{\alpha^Q} \in \IQbar\ssm\{1\}
\end{equation*}
and content ourselves by bounding
\begin{equation}
\label{eq:infiniteplace}
  |x|_w = \left| \beta - 1\right|_w
  |\alpha|_w^Q \le 
\left| \beta - 1\right|_w
\max\{1,|\alpha|_w\}^Q.
\end{equation}

We split the sum (\ref{eq:PF2}) up into the  finite
places in $G v$, the remaining finite places,  and the infinite places. 
The estimates (\ref{eq:sigmainG}), (\ref{eq:finiteplace}), and
(\ref{eq:infiniteplace})  together with the product formula
(\ref{eq:PF2}) yield
\begin{alignat}1
\label{eq:prodformula}
  0 \le 
&\sum_{w \in G v}
d_w \log(p^{-1}) \\
\nonumber
&+
\sum_{w\nmid \infty}
d_w Q \log(\max\{1,|\psi(\alpha)|_w\}\max\{1,|\alpha|_w\}) \\
\nonumber
&+ \sum_{w|\infty}
d_w \left(\log \left|\beta-1\right|_{w}
+Q\log\max\{1,|\alpha|_w\}\right).
\end{alignat}

Moreover,  all local degrees $d_w$ with $w \in Gv$ equal $d_v$. 
So the sum $\sum_{w\in Gv} d_w \log(p^{-1}) =d_v \log(p^{-1})  \#Gv $ is at most  $-[\IQ(N):\IQ](\log
p)/p^4$ by Lemma \ref{lem:largeorbit}.
We use this estimate together with  (\ref{eq:localdegrees}) and (\ref{eq:prodformula}) 
to obtain
\begin{equation*}
  0 \le
-\frac{\log p}{p^4}
+ \frac{1}{[\IQ(N):\IQ]}\left(\sum_{w|\infty}
d_w
\log\left|\beta-1\right|_w\right)
 + Q\height{\psi(\alpha)} + Q\height{\alpha}
\end{equation*}
after dividing by $[\IQ(N):\IQ]$.
The normalized sum over the infinite places 
is the normalized sum over the field embeddings found in
(\ref{eq:ptorsionineq}).

Inequality (\ref{eq:ptorsionineq}) follows from
$\height{\psi(\alpha)}=\height{\alpha}$ and $Q\le p^4$. 
Basic height properties
yield $\height{\beta}\le \height{\psi(\alpha^Q)}
+\height{\alpha^Q}\le 2Q \height{\alpha} \le 2p^4 \height{\alpha}$.

By construction we certainly have $\beta\not=0,1$ and it remains to show
that $\beta$ is not a root of unity. 
If we assume the contrary, then
$\psi(\alpha)/\alpha$ is a root of unity too. 
 Lemma \ref{lem:qpower}  implies
$(\psi(\alpha)/\alpha)^Q=1$, but this contradicts the choice of $\psi$.
\end{proof}

\section{Descending Along $p^n$-Torsion}
\label{sec:descending}
Let $E$ be an elliptic curve defined over $\IQ$
and let $p\ge 5$ be a prime 
with $q=p^2$. We recall the conventions
(\ref{eq:easenotation1}) and (\ref{eq:easenotation2}).

Lemma \ref{lem:localdescent} below is our main tool in the descent argument. 
Given an element in $\IQ(p^nM)$ it allows us to decrease $n$ 
 under certain
circumstances and work in the smaller field $\IQ(p^{n-1}M)$.
The proof involves the  group theory of $\gl{\IF_p}{2}$.
 We thus begin by  recalling some  
facts and  by proving a technical lemma. 

We identify  $\IF_p$ with the scalar matrices in $\mat{\IF_p}{2}$
and consider the latter as an $\IF_p$-algebra.
A subgroup of $\gl{\IF_p}{2}$ is called a non-split Cartan subgroup if it
is the multiplicative group of an $\IF_p$-subalgebra of
$\mat{\IF_p}{2}$
that is a field with $q$ elements.
A non-split Cartan subgroup is cyclic of  order $q-1$.

Conversely, if $G\subset\gl{\IF_p}{2}$ is a cyclic subgroup of order
$q-1$, then it is a non-split Cartan subgroup. Indeed, if $\theta$ is a
generator,  then the Theorem of Cayley-Hamilton implies
that $G$ is contained in the commutative
 $\IF_p$-subalgebra $\IF_p + \IF_p \theta \subset \mat{\IF_p}{2}$.
Now $\theta\not\in\IF_p$, so
counting elements yields
$G = (\IF_p+\IF_p\theta)\ssm\{0\}$ and hence $G$ is a 
non-split Cartan subgroup.


\begin{lemma}
  \label{lem:grouplemma}
Let $G$ be a non-split Cartan subgroup
of $ \gl{\IF_p}{2}$. The set 
\begin{equation}
\label{eq:generators}
  \{ hgh^{-1};\,\, g\in G\text{ and }h\in \gl{\IF_p}{2}\}. 
\end{equation}
has cardinality strictly greater than $p^3$ and  generates $\gl{\IF_p}{2}$
as a group.
\end{lemma}
\begin{proof}
The normalizer of $G$ has cardinality
$2(q-1)$ by Section 2.2 \cite{Serre:Galois}. 
Therefore, the orbit of $G$ under the action of $\gl{\IF_p}{2}$
 by conjugation has
cardinality
$\# \gl{\IF_p}{2}/(2(q-1))$.
Conjugating  $G$ by an element of 
$\gl{\IF_p}{2}$ gives again a non-split Cartan subgroup. 
If $G'$ is a conjugate distinct from $G$, then
$0\cup (G\cap G')$ is an $\IF_p$-subalgebra of $\{0\}\cup G$
and $\{0\}\cup G'$ that has cardinality strictly less than $q$. 
So  $\{0\}\cup (G\cap G')$ has cardinality $p$ since it 
contains the scalar matrices.

The set (\ref{eq:generators}) equals the union of all elements in the
orbit of $G$.
Each orbit element contributes at least $q-p$ elements. So the 
 cardinality of (\ref{eq:generators}) is at least
\begin{equation}
\label{eq:cardestimate}
   \frac{q-p}{2(q-1)} \# \gl{\IF_p}{2}= 
\frac{(p-1)^2 p^2}{2} > p^3
\end{equation}
since $p\ge 5$.
 
 The subgroup of $\gl{\IF_p}{2}$
 generated by (\ref{eq:generators}) contains the non-split Cartan
 subgroup $G$. By Serre's Proposition 17 \cite{Serre:Galois}
it is either $\gl{\IF_p}{2}$ or has cardinality at most $p(p-1)^2$.  But the second alternative is impossible because of
(\ref{eq:cardestimate}). 
\end{proof}

Let $N\in\IN$ with $N=p^nM$ where $n\ge 0$ and $M\ge 1$ are integers
and $p\nmid M$. 
We recall the convention made in Section \ref{ssec:local} and
  identify  $\gal{\IQ_q(N)/\IQ_q}$ with a subgroup 
of $\gal{\IQ(N)/\IQ}$.

\begin{lemma}
\label{lem:localdescent}
Suppose $E$ and  $p$ satisfy (P1) and (P2).
We assume  $p|N$.
\begin{enumerate}
\item [(i)]
The subgroup of $\gal{\IQ(N)/\IQ}$ generated by the conjugates
of $\gal{\IQ_q(N)/\IQ_q(N/p)}$ equals
$\gal{\IQ(N)/\IQ(N/p)}$.
\item [(ii)]
If $\alpha\in\IQ(N)$
with  $\sigma(\alpha) \in \IQ_q(N/p)$ for all 
$\sigma\in\gal{\IQ(N)/\IQ}$, then 
$\alpha\in \IQ(N/p)$. 
\end{enumerate}
\end{lemma}
\begin{proof}
For the proof we  abbreviate
\begin{equation*}
G =  \gal{\IQ_q(N)/\IQ_q(N/p)}.  
\end{equation*}

The hypothesis on $\alpha$ in (ii)  implies
\begin{equation*}
   \sigma\psi\sigma^{-1} \in \gal{\IQ(N)/\IQ(N/p)(\alpha)}
\quad\text{for all}\quad
\sigma \in \gal{\IQ(N)/\IQ}\quad\text{and all}\quad
\psi \in G.
\end{equation*}
Therefore, part (ii) follows immediately from part (i) which we
proceed to prove.

Let $H$ be the subgroup of $\gal{\IQ(N)/\IQ}$ generated
by $\sigma\psi\sigma^{-1}$ where $\sigma$ varies over
$\gal{\IQ(N)/\IQ}$ and $\psi$ varies over $G$. Then
\begin{equation}
\label{eq:Hincluded}
H\subset\gal{\IQ(N)/\IQ(N/p)}  
\end{equation}
 and our task is to show equality. 

It is convenient to fix  isomorphisms
$E[p]\cong (\IZ/p\IZ)^2$ and $E[p^n]\cong (\IZ/p^n\IZ)^2$
that are compatible with the natural inclusion
$E[p]\subset E[p^n]$. We will identify
$\aut{E[p]}$ and $\aut{E[p^n]}$ with
$\gl{\IF_p}{2}$ and $\gl{\IZ/p^n\IZ}{2}$, respectively.
There are two natural Galois representations
\begin{equation*}
  \widetilde \rho : \gal{\IQ(N)/\IQ}\rightarrow \gl{\IF_p}{2}
\quad\text{and}\quad
   \rho : \gal{\IQ(N)/\IQ}\rightarrow \gl{\IZ/p^n\IZ}{2}.
\end{equation*}
They fit into the commutative diagram
\begin{equation}
\label{cd:restriction}
\if \usexypic 1
    \xymatrix{
& \gal{\IQ(N)/\IQ}\ar[d] \ar[r]^{\rho} \ar[dr]_{\widetilde \rho}&  \gl{\IZ/p^n\IZ}{2}\ar[d] \\
 & \gal{\IQ(p)/\IQ} \ar[r]  &  \gl{\IF_p}{2}
}  
\else
\text{(Some diagram)}
\fi
\end{equation}
 where the right vertical arrow is
 the natural surjection and the left 
vertical arrow is induced
by the restriction map. 



Let us continue the proof by splitting up into two cases.

First say $n=1$. Then
 $G$ is cyclic of order $q-1$ by Lemma \ref{lem:galoisprops}(iv).
The same holds for $\rho(G)$ because $\rho|_G$ is injective.
Therefore, $\rho(G)$ is a non-split Cartan subgroup
 of $\gl{\IF_p}{2}$.


By property (P1) the image of
 $\rho=\widetilde{\rho}$ is $\gl{\IF_p}{2}$.
We apply  Lemma \ref{lem:grouplemma} to $\rho(G)\subset\gl{\IF_p}{2}$
and use the fact that $H$ is generated by conjugating
$G$ 
 to obtain
\begin{equation}
\label{eq:rhoH}
  \rho(H) = \gl{\IF_p}{2}.
\end{equation}

The restriction map induces an injective homomorphism
 \begin{equation*}
\gal{\IQ(N)/\IQ(N/p)} \hookrightarrow
\gal{\IQ(p)/\IQ}.
 \end{equation*}
In particular, we get the second inequality in
 \begin{alignat}1
\label{eq:ineqchain}   
\# H 
\le \#\gal{\IQ(N)/\IQ(N/p)} 
\le \# \gal{\IQ(p)/\IQ}. 
 \end{alignat}
the first one follows from (\ref{eq:Hincluded}).
But $\# H \ge \#\gl{\IF_p}{2}$ by (\ref{eq:rhoH})
and thus $\# H \ge \#\gal{\IQ(p)/\IQ}$. 
The chain of inequalities (\ref{eq:ineqchain}) is actually a chain of
equalities. So part (i) of the lemma holds for $n=1$. 



Now we turn to $n\ge 2$.
If $\sigma\in \gal{\IQ(N)/\IQ(N/p)}$ then $\rho(\sigma)$ is  represented by
 $1+p^{n-1}\mathcal{L}'(\sigma)$ 
with $\mathcal{L}'(\sigma)\in \mat{\IZ}{2}$. 
Moreover, $\mathcal{L}'(\sigma)$ is well-defined modulo 
$p\mat{\IZ}{2}$. We obtain a  ``logarithm''
 $\mathcal L: \gal{\IQ(N)/\IQ(N/p)} \rightarrow \mat{\IF_p}{2}$.
The name is justified since if
 $\sigma_1,\sigma_2\in \gal{\IQ(N)/\IQ(N/p)}$, then
\begin{equation*}
 \rho(\sigma_1\sigma_2) \equiv (1+p^{n-1}\mathcal{L}(\sigma_1))(1+p^{n-1}\mathcal{L}(\sigma_2))
\equiv 1 + p^{n-1}(\mathcal{L}(\sigma_1)+\mathcal{L}(\sigma_2)) 
\mod p^n \mat{\IZ}{2}
\end{equation*}
because $n\ge 2$.
So $\mathcal{L}(\sigma_1\sigma_2) = \mathcal L
(\sigma_1)+\mathcal  L(\sigma_2)$ and $\mathcal L$ is thus a
 group homomorphism. It is easily seen to be injective and so we find
\begin{equation}
\label{eq:degreebound}
  [\IQ(N):\IQ(N/p)] \le \#\mat{\IF_p}{2} = p^4.
\end{equation}


 If $\sigma\in \gal{\IQ(N)/\IQ}$ and 
  $\psi \in G$ then   $\sigma\psi\sigma^{-1} \in 
 \gal{\IQ(N)/\IQ(N/p)}$ and
 a short calculation gives
\begin{equation*}
 \rho(\sigma\psi\sigma^{-1}) \equiv 1+ p^{n-1} 
\widetilde{\rho}(\sigma)\mathcal L'(\psi) \widetilde{\rho}(\sigma)^{-1}
\mod p^n\mat{\IZ}{2}. 
\end{equation*}
So
\begin{equation}
\label{eq:ABaction}
\mathcal L(\sigma\psi\sigma^{-1}) = \widetilde\rho(\sigma) \mathcal L(\psi) \widetilde\rho(\sigma)^{-1}.  
\end{equation}

By Lemma \ref{lem:galoisprops}(iv) $G$ has order $p^2$, so
 $\# \mathcal L(G) = p^2$. In particular, $\mathcal L(G)$ contains a non-scalar matrix
$\theta$. 
One consequence of  Lubin-Tate theory, cf. Lemma \ref{lem:LT}(iii),
is that $\rho(G)$ contains 
all scalar matrices in $ \gl{\IZ/p^n\IZ}{2}$. 
Tracing through the definition of $\mathcal L$ this means that
$\mathcal L(G)$
contains the scalar matrices  $\IF_p\subset \mat{\IF_p}{2}$. 
Since $\mathcal L(G)$ is a subgroup of $\mat{\IF_p}{2}$ we find
$\mathcal L(G) = \IF_p + \IF_p \theta$.
By the Theorem of Cayley-Hamilton 
$\theta^2 \in \mathcal L(G)$, so $\mathcal L(G)$
is a commutative $\IF_p$-algebra.

Next we claim that $\theta$ has no eigenvalues in $\IF_p$. 
We recall that $\IQ_q(N)/\IQ_q(M)$ is abelian, cf. Lemma 
\ref{lem:galoisprops}(iii).
So all matrices in $\widetilde\rho(\gal{\IQ_q(N)/\IQ_q(M)})$
commute with $\theta$ by  (\ref{eq:ABaction}).
Therefore,
matrices in 
 the image of $\gal{\IQ_q(p)/\IQ_q}$ under the bottom arrow of
(\ref{cd:restriction}) commute with $\theta$ too.
The said arrow is injective and
we know from Lemma \ref{lem:LT}(i)
that $\gal{\IQ_q(p)/\IQ_q}$ has order $q-1$. 
So $q-1$ divides the order of the
 centralizer of $\theta$ in $\gl{\IF_p}{2}$. If $\theta$ were to have an eigenvalue
 in $\IF_p$, then it would be conjugate, over $\IF_p$, to either
 \begin{equation*}
   \left(
   \begin{array}{cc}
     \phi & 0 \\ 0& \mu
   \end{array}
\right) \quad\text{or}\quad
   \left(
   \begin{array}{cc}
     \phi & 1 \\ 0& \phi
   \end{array}
\right)
 \end{equation*}
for some $\phi,\mu\in \IF_p$. 
The only matrices listed above having centralizer of order divisible
by $q-1$ are the scalar matrices. 
This
contradicts our choice of $\theta$.

Since $\theta$ has no eigenvalues in $\IF_p$ we deduce
$\mathcal L(G)^\times = \mathcal L(G)\ssm\{0\}$. Hence $\mathcal L(G)$ is a field with $q$
elements and 
 $\mathcal L(G)^\times$ is a non-split Cartan
subgroup  of $\gl{\IF_p}{2}$.  

We recall 
that by (P2)
the image of $\widetilde\rho$ is $\gl{\IF_p}{2}$.
So the definition of $H$ 
 and (\ref{eq:ABaction}) imply that  conjugating
a matrix in $\mathcal L(G)$ by any element of $\gl{\IF_p}{2}$  stays
within $\mathcal L(H)$. 
We apply Lemma \ref{lem:grouplemma} to the subgroup $\mathcal L(G)^\times$
and deduce
$\#\mathcal L(H) > p^3$.
But $\mathcal L(H)$  is a subgroup of
$\mat{\IF_p}{2}$. So its cardinality must be $p^4$.


The conclusion of the case $n\ge 2$ is  similar to the case $n=1$:
we have
\begin{alignat}1
\label{eq:ineqchain2}
\# H
&\le  \#\gal{\IQ(N)/\IQ(N/p)} \le p^4
\end{alignat}
where we  used (\ref{eq:Hincluded}) and (\ref{eq:degreebound}).
 But $\# H \ge \#\mathcal L(H) =  p^4$ by the previous paragraph.
As above this  implies  equality throughout
(\ref{eq:ineqchain2}) and the proof of part (i) is complete. 
\end{proof} 

Using this last lemma we can strengthen Lemma \ref{lem:ptorsion}
to cover  the tamely ramified case, i.e. for  algebraic
numbers  in $\IQ(N)$ when $p^2\nmid N$.

\begin{lemma}
  \label{lem:pM}
Suppose $E$ and $p$ satisfy (P1) and (P2).
We assume   $p^2 \nmid N$.
If $\alpha\in\IQ(N) \ssm \mu_\infty$
is non-zero, there exists a non-zero
$\beta\in \IQbar \ssm\mu_\infty$ with 
$\height{\beta}\le 2p^4\height{\alpha}$ and
  \begin{equation}
\label{eq:ptorsionineq2}
\height{\alpha} + 
\max\left\{0,\frac{1}{[\IQ(\beta):\IQ]}
\sum_{\tau}
\log |\tau(\beta)-1|\right\}
\ge \frac{\log p}{2p^8}
  \end{equation}
where the sum runs over all field embeddings $\tau : \IQ(\beta)\rightarrow \IC$.
\end{lemma}
\begin{proof}
For brevity we write $Q = Q(1) = (q-1)q$.
It is no restriction to assume $p|N$. 
If there is  $\sigma\in\gal{\IQ(N)/\IQ}$
with $\sigma(\alpha)^Q \not\in \IQ_q(N/p)$ then we may apply
 Lemma \ref{lem:ptorsion} to $\sigma(\alpha)$. The current
lemma follows because
 $\height{\sigma(\alpha)}=\height{\alpha}$. 

Conversely, if $\sigma(\alpha^Q) \in \IQ_q(N/p)$ 
for all $\sigma$, then
Lemma \ref{lem:localdescent}(ii) implies $\alpha^Q \in \IQ(N/p)$.

But $N/p$ and $p$ are coprime by hypothesis. 
We can refer to the unramified case treated in Lemma
 \ref{lem:primetoptorsion} to deal with $\alpha^Q$. 
Clearly $\alpha^Q$ is non-zero and not a root of unity. So 
\begin{equation*}
  \height{\alpha^Q} \ge \frac{\log(p/2)}{p^2+1}.
\end{equation*}
Basic height properties imply $\height{\alpha^Q} = Q\height{\alpha}
=(p^2-1)p^2\height{\alpha}$, hence
\begin{equation*}
  \height{\alpha}\ge \frac{\log(p/2)}{p^2(p^4-1)}
\ge \frac{\log(p/2)}{p^6}.
\end{equation*}
This lower bound is  better than (\ref{eq:ptorsionineq2}) since $p\ge 5$.
The current lemma follows with $\beta=\alpha$.
\end{proof}

Now we will construct 
 a useful automorphism of $\IQ(N)/\IQ$.

\begin{lemma}
\label{lem:center}
Suppose $E$ and $p$ satisfy (P1). 
Let $n\ge 0$ be the greatest
integer with $p^n|N$.
 There exists $\sigma\in \gal{\IQ_q(N)/\IQ_q}$, lying 
 in the center of
    $\gal{\IQ(N)/\IQ}$,
such that $\sigma(\zeta)=\zeta^4$
for all $\zeta\in \mu_{p^n}$. Moreover, $\sigma$ acts on $E[p^n]$ as
multiplication by $2$. 
\end{lemma}
Before we prove this lemma, let us recall that 
$\IQ_q(N)$ contains $\mu_{p^n}$ by Lemma \ref{lem:rootofone}.
\begin{proof}
Since $p$ is odd,  Lemma \ref{lem:LT}(iii) implies that
there is $\sigma' \in \gal{\IQ_q(p^n)/\IQ_q}$ which acts
on $E[p^n]$ as multiplication by $2$.

By properties of the Weil pairing we see that
 $\sigma'$ acts as $\zeta\mapsto
\zeta^4$ on the roots of unity of order dividing $p^n$;
bilinearity of the Weil pairing is responsible for 
  $4=2^2$ in the exponent.

By Lemma \ref{lem:galoisprops}(iii) the automorphism
$\sigma'$ lifts uniquely to
$\sigma \in \gal{\IQ_q(N)/\IQ_q(M)}$.

Taking the sum of points  gives an isomorphism between
 $E[p^n]\times E[M]$ and $E[N]$ which is compatible with the
 action of $\gal{\IQ(N)/\IQ}$. Since $\sigma$  
 acts    as multiplication by $2$ on $E[p^n]$ and
trivially on $E[M]$, it must lie in the center of
$\gal{\IQ(N)/\IQ}$. 
\end{proof}

In the next proposition we fix the auxiliary prime
$p$ which has accompanied us until now. Its proof contains a Kummerian descent  reminiscent to  one
 used by Amoroso and Zannier \cite{AZ:unifrel}.

\begin{proposition}
\label{prop:prebogo}
Suppose $E$ does not have complex multiplication.
  There exists a constant $c>0$ depending only on $E$ with the
  following property. If $\alpha\in \IQ(\tors{E})\ssm\mu_\infty$
 is non-zero, there is a non-zero $\beta\in\IQbar\ssm\mu_\infty$
with $\height{\beta}\le c^{-1} \height{\alpha}$ and 
\begin{equation}
\label{eq:prebogo}
   \height{\alpha} + \max\left\{0,\frac{1}{[\IQ(\beta):\IQ]}
\sum_{\tau:\IQ(\beta)\rightarrow \IC}
\log |\tau(\beta)-1|\right\} \ge c.
\end{equation}
\end{proposition}
\begin{proof}
Since $E$ does not have complex multiplication, its $j$-invariant is
not $0$ or $1728$. 
So the reduction of $E$ at $p$ is an elliptic curve with $j$-invariant
not among $\{0,1728\}$ for all but finitely many primes $p$.
By a theorem of Serre \cite{Serre:Galois} 
all but finitely many of these $p$ satisfy (P2), that is,  
 the representation
$\gal{\IQbar/\IQ}\rightarrow \aut{E[p]}$ is surjective.
Elkies \cite{Elkies} showed  that $E$ has good supersingular reduction at infinitely
many primes. 
We may thus fix a prime $p\ge 5$ satisfying both (P1) and (P2) and
 set $q=p^2$. 

Let $\alpha$ be as in the hypothesis.
Then $\alpha\in\IQ(N)$ for some 
$N=p^n M $ with $M\in\IN$ coprime to $p$ and $n$ a non-negative integer.

We take $\sigma_4\in \gal{\IQ_q(N)/\IQ_q}$ 
as in Lemma \ref{lem:center}. 
We define
\begin{equation}
\label{eq:lambda}
\gamma =   \frac{\sigma_4(\alpha)}{\alpha^4}
\in \IQ(N).
\end{equation}
Basic height inequalities imply
\begin{equation}
\label{eq:hlambdabound}
  \height{\gamma}\le \height{\sigma_4(\alpha)}+\height{\alpha^4}
= 5\height{\alpha}.
\end{equation}

There is a least integer
$n' \ge 0$ such that  $\sigma(\gamma)\in \IQ_q(p^{n'}M)$
for all $\sigma\in \gal{\IQ(N)/\IQ}$.
It satisfies $n'\le n$ and
 Lemma \ref{lem:localdescent} implies $\gamma\in \IQ(p^{n'}M)$.

Let us first suppose $n'\le 1$, so 
$\gamma\in \IQ(pM)$. 
We want to apply Lemma \ref{lem:pM}, so let us confirm that
 $\gamma\not=0$ is not a root of unity.
Otherwise we would have
$4\height{\alpha}=\height{\alpha^4}=\height{\gamma\alpha^4}=\height{\sigma_4(\alpha)}=
\height{\alpha}$ by the basic height properties.
So $\height{\alpha}=0$. Kronecker's Theorem implies $\alpha=0$
or $\alpha\in\mu_\infty$. This 
 contradicts our assumption on $\alpha$.
Hence Lemma \ref{lem:pM} provides a non-zero
$\beta\in\IQbar\ssm\mu_\infty$ with
$\height{\beta} \le 2p^4 \height{\gamma}$
and
\begin{equation*}
  \height{\alpha} + 
\max\left\{0,\frac{1}{[\IQ(\beta):\IQ]}
\sum_{\tau}
\log |\tau(\beta)-1|\right\}
\ge \frac{\log p}{2p^8}
\end{equation*}
The bound (\ref{eq:hlambdabound}) gives
$\height{\beta} \le 10p^4 \height{\alpha}$.
Moreover, we can use (\ref{eq:hlambdabound})
  to deduce (\ref{eq:prebogo})
 with a constant
 $c$ depending only on $p$.

Hence Proposition \ref{prop:prebogo} follows if $n'\le 1$ and we will now assume
$n'\ge 2$.

By minimality of $n'$ there is $\sigma\in \gal{\IQ(N)/\IQ}$
with $\sigma(\gamma)\not\in \IQ_q(p^{n'-1}M)$. 
 We abbreviate  $\alpha'=\sigma(\alpha)$ and $\gamma'=\sigma(\gamma)$.
We 
apply $\sigma$ to (\ref{eq:lambda}) and obtain
\begin{equation}
\label{eq:lambda2}
\gamma' =   \frac{\sigma(\sigma_4(\alpha))}{\sigma(\alpha)^4}=
\frac{\sigma_4(\sigma(\alpha))}{\sigma(\alpha)^4}
=\frac{\sigma_4(\alpha')}{\alpha'^4}
\end{equation}
since $\sigma_{4}$ lies in the center of $\gal{\IQ(N)/\IQ}$.

Next we would like to apply Lemma \ref{lem:ptorsion} to $\gamma'$.
In order to do this we need to verify the hypotheses.
Note that we have $Q(n')=q$ since $n'\ge 2$,
so we must prove
  $\gamma'^q \not\in
\IQ_q(p^{n'-1}M)$.
We now assume the contrary and will soon arrive at a
 contradiction.


Since $\gamma'\not\in \IQ_q(p^{n'-1}M)$ there is
$\psi \in \gal{\IQ_q(N)/\IQ_q(p^{n'-1}M)}$ 
with $\psi(\gamma')\not=\gamma'$.
However, $\psi(\gamma'^q) = \gamma'^q$ and so
\begin{equation}
\label{eq:lambdanoteq}
 \psi(\gamma') = \xi\gamma'\quad\text{with}\quad \xi^q=1
\quad\text{while}\quad \xi\not=1.
\end{equation}
We identify $\psi$ with its restriction to $\IQ(N)$, apply
it to (\ref{eq:lambda2}), and obtain
\begin{equation*}
\xi \gamma' =
\frac{\sigma_4(\psi(\alpha'))}{\psi(\alpha')^4}
\end{equation*}
 having used  that $\psi$ commutes with $\sigma_{4}$. 
We define  $\eta =
\psi(\alpha')/\alpha'\not=0$ and get
\begin{equation*}
  \xi = \frac{\sigma_4(\eta)}{\eta^{4}}.
\end{equation*}
Basic height properties and the fact that $\xi$ is a root of unity
give
$4\height{\eta}=\height{\eta^4} = \height{\xi\eta^4} =
 \height{\sigma_4(\eta)}
 = \height{\eta}$
and as usual,  $\height{\eta}=0$. So $\eta$ is a
root of unity by Kronecker's Theorem.

We have just shown  $\psi(\alpha')/\alpha'\in\mu_\infty$.  
We  fix $\widetilde M\in\IN$ coprime to $p$ such that
$(\psi(\alpha')/\alpha')^{\widetilde M} \in \mu_{p^\infty}$.
Lemma \ref{lem:rootofone} 
implies $(\psi(\alpha')/\alpha')^{\widetilde M} \in \mu_{p^n}$.
So $\sigma_4$ raises this element to the fourth power, hence
\begin{equation*}
  \sigma_4\left(\frac{\psi(\alpha')}{\alpha'}\right)=
 \xi' \left(\frac{\psi(\alpha')}{\alpha'}\right)^{4}
\end{equation*}
with ${\xi'}^{\widetilde M}=1$. We rearrange 
this expression to obtain
\begin{equation*}
  \frac{\sigma_4(\psi(\alpha'))}{\psi(\alpha')^{4}}
= \xi' \frac{\sigma_4(\alpha')}{\alpha'^{4}} = \xi'\gamma'
\end{equation*}
using (\ref{eq:lambda2}). 
Applying again  the fact that
 $\psi$ and $\sigma_4$ commute gives 
\begin{equation*}
\psi(\gamma')=  \frac{\psi(\sigma_4(\alpha'))}{\psi(\alpha')^4}
=  \frac{\sigma_4(\psi(\alpha'))}{\psi(\alpha')^4}
= \xi' \gamma'.
\end{equation*}
We recall (\ref{eq:lambdanoteq}) and find $\xi'=\xi$,
so $\xi^{\widetilde M} = \xi^{q}=1$.
But $\widetilde M$ and $q=p^2$ are coprime, hence $\xi=1$.
This contradicts (\ref{eq:lambdanoteq}).

So we must have
$\gamma'^q\not\in\IQ_q(p^{n'-1}M)$ and
 Lemma
\ref{lem:ptorsion} yields
a lower bound for $\height{\gamma'}$ involving 
a non-zero $\beta\in\IQbar\ssm\mu_\infty$ with $\height{\beta}\le
2p^4\height{\gamma'}$. 
We have $\height{\gamma'}=\height{\gamma}\le 5\height{\alpha}$
by (\ref{eq:hlambdabound})
so $\height{\beta}\le  10p^4\height{\alpha}$.
Comparing the same upper bound  for $\height{\gamma'}$  with the lower bound
provided from Lemma \ref{lem:ptorsion} completes the proof. 
\end{proof}

\section{Equidistribution}
\label{sec:equi}

After an extensive analysis of the places above a fixed prime $p$, we
 turn our
attention to the infinite places.

Let us suppose for the moment that we are in the situation of the
Proposition \ref{prop:prebogo}.
The normalized sum over $\tau$ is by definition of the height
at most $\height{\beta-1}$. So the bound (\ref{eq:prebogo}) entails
\begin{equation*}
  \height{\alpha}+\height{\beta-1}\ge c.
\end{equation*}
Basic height inequalities show 
\begin{equation*}
\height{\beta-1}\le \height{\beta} + \log 2.
\end{equation*}
Indeed, $\log 2$ originates from the  triangle inequality
\begin{equation*}
  \log |\beta-1|_v \le \log(|\beta|_v + 1) \le \log\max\{1,|\beta|_v\} + \log 2
\end{equation*}
which holds for any infinite place $v$ of the number field $\IQ(\beta)$.
Our proposition also implies $\height{\beta}\le c^{-1}\height{\alpha}$, so $\beta$
has small height if $\alpha$ does. We find
\begin{equation*}
  (1+c^{-1})\height{\alpha}  + \log 2 \ge c.
\end{equation*}
Unfortunately,  $\log 2$ 
 spoils  the inequality completely; we
 obtain no information on
$\height{\alpha}$. 

What we need is a more refined estimate involving the infinite
places. This is provided by Bilu's Equidistribution Theorem
\cite{Bilu} which takes into account that $\beta$ has small height.
 We state it in a form streamlined  for
our application. 

\begin{theorem}[Bilu]
  Let $\beta_1,\beta_2,\ldots$ be a sequence of non-zero elements of 
$\IQbar\ssm \mu_\infty$ with $\lim_{k\rightarrow \infty}
  \height{\beta_k}=0$.  
If $f:\IC\ssm\{0\}\rightarrow \IR$ is a continuous and bounded
function, then
\begin{equation*}
  \lim_{k\rightarrow\infty}
\frac{1}{[\IQ(\beta_k):\IQ]}\sum_{\tau}
f(\tau(\beta_k)) = \int_0^1 f(e^{2\pi i t})dt
\end{equation*}
where  $\tau$ runs over all field embeddings $\IQ(\beta_k)\rightarrow \IC$.
\end{theorem}

We now prove Theorem \ref{thm:main}.


We suppose first that $E$ has complex multiplication. As
 we have seen in the
introduction, Amoroso and Zannier's result \cite{AZ:relDobrowolski} implies that
$\IQ(\tors{E})$ satisfies the Bogomolov property. 
So let us assume that $E$ does not have complex multiplication.

Our argument is by contradiction. We suppose that
 $\alpha_1,\alpha_2,\dots$ is a sequence of non-zero elements
of $\IQ(\tors{E})\ssm\mu_\infty$ with $\lim_{k\rightarrow\infty} \height{\alpha_k}=0$. 



Let $m\in \IN$ we define a 
 continuous and bounded function
$f_m:\IC\ssm\{0\}\rightarrow \IR$
 by setting
\begin{equation*}
  f_m(z) = \min\{m,\max\{-m,\log|z-1|\}\}
\end{equation*}
for $z\not=1$ and $f_m(1) =-m$.

The sequence of functions $s\mapsto f_m(e^{2\pi i s})$ converges pointwise
to $s\mapsto \log |e^{2\pi i s}-1|$ on $(0,1)$ as $m\rightarrow \infty$.
Clearly, $|f_m(e^{2\pi i s})|\le | \log |e^{2\pi i s}-1||$
and $\int_0^1 |\log|e^{2\pi i s}-1|| ds <\infty$. So
the Dominant Convergence Theorem from analysis implies
\begin{equation*}
  \lim_{m\rightarrow\infty} \int_0^1 f_m(e^{2\pi i s}) ds = 
\int_0^1 \log|e^{2\pi i s}-1|ds. 
\end{equation*}

The latter integral is the logarithmic Mahler measure of the 
polynomial $X-1$. As such, it vanishes by Jensen's Formula. 
So we may fix once and for all  an  $m$ with
\begin{equation}
\label{eq:choosem}
  \int_0^1 f(e^{2\pi i s})ds < \frac c2
\quad\text{and}\quad
\log(1+2e^{-m}) \le  \frac c2
\end{equation}
where $c$ is the positive constant from Proposition 
\ref{prop:prebogo}
and $f=f_m$.

The proposition also gives us a
non-zero $\beta_k\in \IQbar\ssm\mu_\infty$
for each $\alpha_k$
which satisfies
\begin{equation}
\label{eq:weakbound}
  \height{\alpha_k} + \max\left\{0,\frac{1}{[\IQ(\beta_k):\IQ]}
\sum_{\tau:\IQ(\beta_k)\rightarrow\IC}
\log |\tau(\beta_k)-1| \right\} \ge c.
\end{equation}
and 
\begin{equation}
\label{eq:heighttkbound}
 \height{\beta_k} \le \frac{\height{\alpha_k}}{c}.
\end{equation}

We proceed by bounding the sum in (\ref{eq:weakbound}) from above. 
Let $\tau:\IQ(\beta_k)\rightarrow \IC$ be an embedding. 
We write $z=\tau(\beta_k) \in \IC\ssm\{0,1\}$ and split up into cases depending on the 
size of $|z-1|$.

Suppose for the moment that $|z-1| \ge e^m$. Then
$|z|\ge e^m -1 \ge e^m/2$ since $m\ge 1$.
So $|z-1|/|z| \le 1+1/|z| \le 1+2e^{-m}$. Applying the logarithm 
and using (\ref{eq:choosem}) gives
\begin{equation*}
  \log|z-1| \le \log(1+2e^{-m}) + \log |z| \le \frac c2 + \log|z|
\le \frac c2 + \log \max\{1,|z|\}.
\end{equation*}
Because $f(z) = m \ge 0$ we conclude
\begin{equation}
\label{eq:logtkbound}
  \log|\tau(\beta_k)-1| \le \frac c2 + \log\max\{1,|\tau(\beta_k)|\} + f(\tau(\beta_k)).
\end{equation}

The second case is  $|z-1|< e^m$. Then $\log|z-1|\le  \max\{-m,\log|z-1|\}
=f(z)$. 
So   (\ref{eq:logtkbound}) holds as well.

Taking the sum over all field embeddings $\tau:\IQ(\beta_k)\rightarrow\IC$,
applying  (\ref{eq:logtkbound}),
  and dividing by the
degree yields
\begin{equation*}
  \frac{1}{[\IQ(\tau_k):\IQ]}\sum_{\tau}
\log|\tau(\beta_k)-1| 
\le \frac c2 +\height{\beta_k} + 
  \frac{1}{[\IQ(\tau_k):\IQ]} \sum_{\tau}
f(\tau(\beta_k)).
\end{equation*}
Hence (\ref{eq:weakbound}) implies
\begin{equation}
\label{eq:weakbound2}
 \height{\alpha_k} +    
  \max\left\{0,\frac c2 + \height{\beta_k} +\frac{1}{[\IQ(\tau_k):\IQ]} \sum_{\tau}
f(\tau(\beta_k))\right\} \ge c.
\end{equation}

The sequence  $\height{\alpha_1},\height{\alpha_2},\dots$ tends to
zero, hence so does $\height{\beta_1},\height{\beta_2},\ldots$ by
(\ref{eq:heighttkbound}).
 We will apply Bilu's Theorem to $\beta_1,\beta_2,\dots$ and
 the function $f$. 
On letting $k\rightarrow \infty$  the sum
 (\ref{eq:weakbound2}) over the $\tau$ converges to the integral 
$\int_0^1 f(e^{2\pi i s})ds < c/2$ and both terms involving the height
vanish. This is a contradiction.
\qed

\section{Height Lower Bounds on Elliptic Curves}
\label{sec:thm2}

\subsection{The N\'eron-Tate Height}
\label{sec:ntheight}
Let $E$ be an elliptic curve defined over a number field $F$. 
We suppose that $E$ is presented by a short Weierstrass equation.

The N\'eron-Tate height takes a point $A\in E(F)$
to a  real number $\hat h(A)\ge 0$. 
It can be defined either as 
 a sum of local heights or a limit process involving the Weil
 height. We begin with a brief review of the first definition. Say $v$ is a place of $F$
and let $E_v$ be $E$ taken as an elliptic curve defined over $F_v$.  
When working at a fixed place we will assume $F\subset F_v$. 
There is a
local height function $\lambda_v:E(F_v)\ssm\{0\}\rightarrow \IR$, some
of whose properties are discussed below.
These local height functions are defined in Chapter VI
\cite{Silverman:Adv} and they are independent of the chosen Weierstrass equation.
They sum up to give the N\'eron-Tate height
\begin{equation*}
    \hat h(A) = \frac{1}{[F:\IQ]}\sum_{v\text{ place of }F} d_v \lambda_v(A)
  \end{equation*}
for $A\not=0$. We remark that only finitely many terms $\lambda_v(A)$
are non-zero and set $\hat h(0)=0$. 
Let $K$ be a number field containing $F$ and $w$
a place of $K$ extending $v$. Then we may take $F_v\subset K_w$
and we have $\lambda_v = \lambda_w$ on $E(F_v)$. So we obtain a local
height function $\lambda_v:E(\overline{F_v})\ssm\{0\}\rightarrow\IR$
where  $\overline{F_v}$
is an algebraic closure of $F_v$.

Because we are working with a Weierstrass equation any
 $A\in E(F)\ssm\{0\}$ can be expressed as $A=(x,y)$. We set
$\height{A} = \height{x}/2$ and $\height{0} = 0$. 
The definition of the N\'eron-Tate height in terms of local heights is
equivalent to  
\begin{equation*}
  \hat h(A) = \lim_{k\rightarrow\infty}\frac{\height{[2^k](A)}}{4^k}.
\end{equation*}
 We  refer to Chapter VIII, \S 9 \cite{Silverman:AEC} 
for the
  basic properties of the N\'eron-Tate height which follow. 

The N\'eron-Tate height
 does not depend on the number field $F$ over which the
point $A$ is defined. 
We thus obtain a well-defined function $\hat h:E(\overline
F)\rightarrow [0,\infty)$ on any algebraic closure $\overline F$ of
  $F$. 
The elliptic version of Kronecker's Theorem also holds: the
N\'eron-Tate height vanishes precisely on
$\tors{E}$. Moreover, it satisfies the parallelogram equality
\begin{equation*}
  \hat h(A+B) + \hat h(A-B) = 2\hat h(A) + 2\hat h(B)
\end{equation*}
for all $A,B\in E(\overline F)$ as well as
\begin{equation*}
  \hat h(nA) = n^2 \hat h(A)
\end{equation*}
for all  $n\in \IZ$. 
A direct consequence is 
\begin{equation}
\label{eq:torsioninvariance}
 \hat h(A+B) = \hat h (A) 
\quad\text{if $B$ happens to be a torsion point.}
\end{equation}

If $\ell\ge 2$ is a prime number or if $\ell=\infty$ 
it will be convenient to define the partial height function
\begin{equation*}
  \hat h_\ell(A)  = \frac{1}{[F:\IQ]}
\sum_{v|\ell} d_v \lambda_v(A)
\end{equation*}
for $A\in E(F)\ssm\{0\}$.
Then $\hat h_\ell$ extends to a well-defined function $E(\overline
F)\ssm\{0\}\rightarrow \IR$. 
In this notation
\begin{equation*}
  \hat h = \hat h_{\infty} + \hat h_2+ \hat h_3+\cdots.
\end{equation*}

We briefly discuss some relevant equidistribution 
properties of local height functions.
To do this let $v$ be a place of $F$.

Suppose first that $v$ is an infinite place of $F$. 
Up to complex conjugation, $v$ determines a
 field embedding $\sigma_0:F\rightarrow\IC$.
We thus obtain  an elliptic curve  $E_v$ defined
over $\IC$. The local height function $\lambda_v :
E_v(\IC)\ssm\{0\}\rightarrow \IR$ is given explicitly in Theorem VI.3.2
 \cite{Silverman:Adv}. 
There is $\tau\in \IC$ with positive imaginary part $\imag{\tau}$ and
 a complex analytic isomorphism
$\IC/(\IZ+\tau\IZ)\rightarrow E_v(\IC)$ of groups involving the Weierstrass elliptic
function. 
We abbreviate $q = e^{2\pi i \tau}$ and remark $|q|<1$. 
If $A\in E_v(\IC)\ssm\{0\}$ is the image of $z\in \IC$
and $u = e^{2\pi i z}$, then
\begin{equation}
\label{eq:localheightinf}
  \lambda_v(A) = - \frac 12 b_2\left(\frac{\imag{z}}{\imag{\tau}}\right)\log|q|
- \log |1-u| - \sum_{n\ge 1} \log|(1-q^n u)(1-q^n u^{-1})|
\end{equation}
where $b_2 = X^2-X+1/6$ is the second Bernoulli polynomial. 

The group $E_v(\IC)$ endowed with the complex topology is compact. 
 Hence it comes with a unique Haar measure
$\mu_{E,v}$ of total measure $1$.


A sufficiently strong
 analog to Bilu's Equidistribution Theorem is given by 
 Szpiro, Ullmo, and Zhang's Th\'eor\`eme 1.2 \cite{SUZ:equi}
 which we state in simplified form.

\begin{theorem}[Szpiro, Ullmo, Zhang]
\label{thm:SUZ}
We keep the notation above. 
  Let $P_1,P_2,\ldots \in E(\overline F)\ssm \tors{E}$ be a sequence
  of points with
$\lim_{k\rightarrow \infty}\hat h(P_k) = 0$. 
If $f:E_v(\IC)\rightarrow\IR$ is a continuous function, then 
\begin{equation*}
  \lim_{k\rightarrow\infty} \frac{1}{[F(P_k):F]} \sum_{\sigma} f(\sigma(P_k)) = 
\int f \mu_{E,v}
\end{equation*}
where $\sigma$ runs over all field embeddings $\sigma:F(P_k)\rightarrow
\IC$
extending $\sigma_0$. 
\end{theorem}

Now suppose $v$ is a finite place of $F$
where  $E$ 
 has good reduction.
If $A=(x,y)\in E_v(F_v)\ssm\{0\}$,  we have
\begin{equation}
\label{eq:goodreduction}
\lambda_v(A) = \frac 12 \max\{0,  \log |x|_v\}
\end{equation}
by Theorem VI.4.1 \cite{Silverman:Adv}. In particular,
$\lambda_v(A) \ge 0$.

Suppose
 $E_v$ has split multiplicative reduction.
The local
height  can be evaluated using the Tate uniformization. More precisely,
there is $q\in F_v^\times$ with $|q|_v < 1$ and a surjective group
homomorphism $\phi: F_v^\times \rightarrow E(F_v)$ with kernel
$q^\IZ$, the cyclic
group generated by $q$.  
Thus any point $A\in E(F_v)\ssm\{0\}$ is $\phi(u)$ for some
 $u \in F_v^\times \ssm q^\IZ$ with $|q|_v <|u|_v \le 1$.
By Theorem VI.4.2 \cite{Silverman:Adv} we have
\begin{equation}
\label{eq:localheightsmr}
  \lambda_v(A) =  - \frac 12
  b_2\left(\frac{\log|u|_v}{\log |q|_v}\right) \log |q|_v-\log |1-u|_v,
\end{equation}
the non-Archimedean analog of (\ref{eq:localheightinf}).

The Tate uniformization extends to a group homomorphism
$\overline{F_v}^\times \rightarrow E_v(\overline{F_v})$ with Kernel
$q^\IZ$. The expression for $\lambda_v(A)$ above holds for all $A\in
E(\overline{F_v})\ssm\{0\}$. It is evident that
$\lambda_v$ is invariant under the operation of
$\gal{\overline{F_v}/F_v}$. 

The topological group
$\IR /  \IZ$
 is homeomorphic to the unit circle and  thus equipped with the
unique Haar measure $\mu_{\IR/\IZ}$ of total mass $1$.
The preimage under $\phi$ of a point 
$A\in E_v(\overline{F_v})$ determines 
$\log|u|_v\in \IR$ uniquely up to addition of an
 integral multiple of $\log |q|_v$. Hence the coset
 $\log|u|_v/\log|q|_v + \IZ$
 is a well-defined  element $l_v(A)\in\IR/\IZ$. 

Say $K\subset \overline{F_v}$ 
is a finite extension of $F_v$. Then 
$\log |K^\times| / \log|q|_v+\IZ \subset \IR/\IZ$ 
is in  bijection with the irreducible components of the N\'eron model
of $E$ after a change of base to $K$. Roughly speaking, the set
of these irreducible components becomes the group of torsion points on
$\IR/\IZ$ when $K$ is replaced by the ``limit''
$\overline{F_v}$. 
Let $\overline{F}$ be the algebraic closure of $F$ in $\overline{F_v}$.
 Chambert-Loir's  Theorem implies that the reduction of
 the conjugates of a point in $E(\overline F)$ of  
 small N\'eron-Tate height
are  evenly distributed among  these irreducible
components. His result holds for abelian varieties. But we state it,
according to our needs, for an elliptic curve.


\begin{theorem}[Chambert-Loir, Corollaire 5.5 \cite{CL:mesures}]
\label{thm:CL}
We keep the notation above. 
  Let $P_1,P_2,\ldots \in E(\overline F)\ssm \tors{E}$ be a sequence
  of points with
$\lim_{k\rightarrow \infty}\hat h(P_k) = 0$. 
If $f:\IR/\IZ\rightarrow\IR$ is a continuous function, then 
\begin{equation*}
  \lim_{k\rightarrow\infty} \frac{1}{[F(P_k):F]} \sum_{\sigma}
f(l_v(\sigma(P_k))) = 
\int f \mu_{\IR/\IZ}
\end{equation*}
where $\sigma$ runs over all field embeddings $F(P_k) \rightarrow
\overline{F_v}$
which are the identity on $F$. 
\end{theorem}

The cases when $E_v$ has non-split multiplicative or additive reduction
will not be relevant for our application. 

\subsection{Proof of Theorem \ref{thm:main_ec}}

Let $E$ be an elliptic curve defined over $\IQ$ presented by a short
Weierstrass equation.
In the current section we prove that a non-torsion point with
coordinates in $\IQ(\tors{E})$ cannot have arbitrarily small
N\'eron-Tate height. As already explained in the introduction, the
method of proof is quite similar to the proof that $\IQ(\tors{E})$ has
the Bogomolov property. We proceed by proving a series of lemmas, most
of which have counterparts in previous sections. 

If $p$ is any prime then $E[p^\infty] = \bigcup_{n\ge 0}E[p^n]$ 
denotes the subgroup of
$\tors{E}$ of elements with order a power of $p$. 

We fix some notation used throughout this section. 
Let $p\ge 5$ be a prime  which satisfies properties (P1) and (P2) with
respect to $E$. We set $q=p^2$. 
Let $N$ be a positive integer with $N=p^nM$ where $M\in \IN$ is
coprime to $p$ and $n$ is a non-negative integer. 
It will also be convenient to fix a short Weierstrass equation
for $E$ with
integer coefficients which has good reduction at $p$.

Our first lemma is the analog to Lemma \ref{lem:rootofone}.
We will again use the convention (\ref{eq:easenotation1}) to simplify notation.

\begin{lemma}
\label{lem:torsionpoint}
  We have $E(\IQ_q(N)) \cap E[p^\infty] = E[p^n]$. 
\end{lemma}
\begin{proof}
 The inclusion ``$\supset$'' is obvious. 
So let $T\in E(\IQ_q(N))$ be a torsion point of order
$p^{n'}$. Without loss of generality, we may assume $n'\ge n$ and
$n'\ge 1$. 

By Lemma \ref{lem:LT}(iii) the Galois group
$\gal{\IQ_q(p^{n'})/\IQ_q}$ acts transitively on the
 torsion points of
order $p^{n'}$. Now any conjugate of $T$ over $\IQ_q$ is
again defined over $\IQ_q(N)$. Hence
 we find
$\IQ_q(p^{n'})\subset \IQ_q(N)$.
By
 Lemma \ref{lem:LT}(i)  the ramification index of $\IQ_q(p^{n'})/\IQ_q$
is $(q-1)q^{n'-1}$.
Using the same lemma together with  Lemma \ref{lem:galoisprops}(ii) we find
the ramification index of 
 $\IQ_q(N)/\IQ_q$ to be either $1$ or $(q-1)q^{n-1}$,
depending on whether $n=0$ or $n\ge 1$. 
The first ramification index is at most the second one, so we
 deduce $n'\le n$.
\end{proof}

The next lemma is the elliptic version of Lemma \ref{lem:qpower}. 
We reuse the symbol $Q(n)$ defined in (\ref{eq:defineQn}).

\begin{lemma}
\label{lem:kummer_ec}
  Let us suppose $n\ge 1$. 
If $\psi\in \gal{\IQ_q(N)/\IQ_q(N/p)}$ and 
$A\in E(\IQ_q(N))$ such that $\psi(A) - A \in \tors{E}$, then
\begin{equation*}
  \psi(A)-A \in E[Q(n)].
\end{equation*}
\end{lemma}
\begin{proof}
The order of $B = \psi(A) - A$ is $N'=p^{n'}M'$ for some 
integers $n'\ge 0$ and $M'\ge 1$ with $p\nmid M'$. 

The order of $T=[p^{n'}](B)$ is $M'$ and thus coprime to $p$. 
It follows from Lemma \ref{lem:galoisprops}(ii) that $\IQ_q(N)/\IQ_q(M)$ is totally
ramified. But $T$ is defined over $\IQ_q(N)$ and so 
$\IQ_q(M)(T)/\IQ_q(M)$ is totally ramified as well. 
Now $\IQ_q(M)(T)\subset \IQ_q(MM')$ and
$\IQ_q(MM')/\IQ_q$ is unramified by the same lemma. We conclude that
$T\in E(\IQ_q(M))$. 
In particular, any $\psi$  as in the hypothesis fixes $T$. 

The order of $[M'](B)$ is $p^{n'}$. So Lemma \ref{lem:torsionpoint}
yields $[M'](B)\in E[p^n]$. Therefore $[pM'](B)\in E[p^{n-1}]$
and applying the lemma a second time gives
$[pM'](B) \in E(\IQ_q(N/p))$. Thus $\psi$ fixes $[pM'](B)$ too. 

As in the proof of Lemma \ref{lem:qpower} we deduce 
that $\psi$ fixes $[p](T)$. Following  lines similar to (\ref{eq:ttrick}) we find
$[pt](B)=0$ where $t$ is the order of $\psi$. The proof of this lemma
also concludes similarly using $pt | Q(n)$. 
\end{proof}

As usual we will take $\IQ(N)$ as a subfield of $\IQ_q(N)$
while taking 
(\ref{eq:easenotation2}) into account.
The absolute value
$|\cdot |_p$ on
$\IQ_q(N)$ induces a place $v$ of $\IQ(N)$. 
We let $\widetilde E$ denote the  reduction of $E$ at $p$. We take it
as an elliptic curve defined over $\IF_q$. 
Let $a_q\in \IZ$ be the trace of $q$-Frobenius as in Section
\ref{sec:supersingular}.
By Lemma   \ref{lem:frobscalar} we have $a_q = \pm 2p$.

Next we must translate the two metric lemmas from Section
\ref{sec:localmetric}. 

The first variant deals with the unramified case
and uses again $\varphi_q \in \gal{\IQpunr/\IQ_q}$,
the  lift of Frobenius squared. This field automorphism acts on
 $E(\IQpunr)$.

\begin{lemma}
\label{lem:metric2_ec}
Say $p\nmid N$ and  $A\in E(\IQ_q(N))$. Then
 $A\in E(\IQpunr)$ and if furthermore 
$\varphi_p(A)\not=[a_q/2](A)$,  then
 \begin{equation*}
\lambda_v(\varphi_q(A) - [a_q/2](A))\ge \frac 12 \log p.
 \end{equation*}  
\end{lemma}
\begin{proof}
As in the proof of Lemma \ref{lem:metric2},
 the  first claim follows from Lemma \ref{lem:unramified}. 
By Lemma
  \ref{lem:frobscalar} 
 the $q$-Frobenius endomorphism $\widetilde{\varphi_q}$ of 
$\widetilde E$
 acts as multiplication by $a_q/2$.
 Therefore,
$\varphi_q(A)-[a_q/2](A)$ reduces to $0$. 
Since $E$ has good reduction
at $v$, we may use (\ref{eq:goodreduction})
 to evaluate $\lambda_v(\varphi_q(A)-[a_q/2](A))$. 
The lemma follows since
 $\IQ_q(N)/\IQ_q$ is unramified by Lemma
\ref{lem:unramified}. 
\end{proof}

The second variant deals with the ramified case.

\begin{lemma}
\label{lem:metric1_ec}
If $p| N$ and  $A\in E(\IQ_q(N))$, then
 \begin{equation}
\label{eq:metric1_ec}
\lambda_v(\psi(A) - A) \ge  \frac{\log p}{2(p^2-1)}
 \end{equation}
for all $\psi \in \gal{ \IQ_q(N)/ \IQ_q(N/p)}$
with $\psi(A)\not=A$. 
\end{lemma}
\begin{proof}
  As in the proof of Lemma  \ref{lem:metric1} we find that
$\psi$ lies in the higher ramification group  
$\ram{\IQ_q(N)/\IQ_q}{i}$ with $i=q^{n-1}-1$. Let
  $\mathfrak{P}$ be the maximal ideal of the ring of integers of
  $\IQ_q(N)$. Then $\psi(A)$ and $A$ map to same element on $E$
  reduced modulo $\mathfrak{P}^{q^{n-1}}$. 
Suppose $x$ is the first coordinate of $\psi(A)-A$ in 
our fixed Weierstrass
model of $E$. Then
$\log |x|_p \ge \frac{q^{n-1}}{e} \log p$ with $e$ the ramification
index of $\IQ_q(N)/\IQ_q$. 
By Lemmas \ref{lem:LT}(i) and \ref{lem:galoisprops}(ii)
we have $e = (q-1)q^{n-1}$. Now (\ref{eq:metric1_ec}) follows from
(\ref{eq:goodreduction}) and $q=p^2$. 
\end{proof}

According to the blueprint of Theorem \ref{thm:main}'s proof 
the next step should be to
imitate  Lemma \ref{lem:primetoptorsion} and obtain a height
lower bound in the
unramified case. We postpone this task until later and for now only
obtain a lower bound for the partial height function $\hat h_p$. 
The N\'eron-Tate height
 is the sum of all partial height functions, but the
partial height functions at primes of bad reduction or at $\infty$ may
take negative values. So a lower bound for $\hat h_p$ does not
directly imply a lower bound for $\hat h$. 

\begin{lemma}
  \label{lem:primetoptorsion_ec}
We assume $p\nmid N$.
 If $A\in E(\IQ(N))\ssm\tors{E}$
there is a non-torsion point $B\in E(\IQbar)$ with $\hat h(B) \le 2(p^2+1) \hat h(A)$
such that
\begin{equation*}
  \hat h_p(B)\ge \frac 12 \log p.
\end{equation*}
\end{lemma}
\begin{proof}
We set $B = \varphi_q(A) - [a_q/2](A)$
and remark that $B$ is not a torsion point. 
Indeed, otherwise we would
have $\hat h (A) = \hat h(\varphi_q(A)) = \hat h([a_q/2](A)) = p^2
\hat h(A)$ by properties of the N\'eron-Tate height and since $a_q/2
=\pm p$. This implies
$\hat h(A) = 0$ and so $A\in \tors{E}$ by Kronecker's Theorem,
contradicting  our hypothesis.

The parallelogram equality implies
\begin{equation*}
  \hat h(B) \le \hat h(\varphi_q(A)-[a_q/2](A)) 
+\hat h(\varphi_q(A)+[a_q/2](A)) 
= 2\hat h(\varphi_q(A)) + 2 \hat h([a_q/2](A))
\end{equation*}
and we deduce $\hat h(B)\le 2(p^2+1)\hat h(A)$, as desired.

As in the proof of Lemma \ref{lem:primetoptorsion} we see that
the restriction of
$\varphi_q$ lies in the center of
$\gal{\IQ(N)/\IQ}$. 
This observation together with (\ref{eq:goodreduction}) yields
\begin{equation*}
\lambda_{\sigma^{-1}v}(B)=
\lambda_v(\sigma(\varphi_q(A)) -\sigma([a_q/2](A))) = 
  \lambda_v(\varphi_q(\sigma(A)) -[a_q/2](\sigma(A)))
\end{equation*}
and $\varphi_q(\sigma(A)) \not= [a_q/2](\sigma(A))$. 
So $\lambda_{\sigma^{-1}v}(B) \ge (\log p)/2$
by Lemma \ref{lem:metric2_ec}.
As $\sigma$ varies over the elements of $\gal{\IQ(N)/\IQ}$ we
obtain any place above $p$ as some $\sigma^{-1}v$. 

We recall  (\ref{eq:localdegrees}). 
Summing up the local heights over all places above $p$
 with the correct
multiplicities and dividing by $[\IQ(N):\IQ]$  yields
\begin{equation*}
\hat h_p(B) \ge \frac 12 \log p. 
 \qedhere
\end{equation*}
\end{proof}

Now we begin tackling the unramified case.

\begin{lemma}
\label{lem:ptorsion_ec}
 We assume $p|N$ and
let $n\ge 1$ be the greatest
integer with $p^n|N$.
If  $A\in E(\IQ(N))$   satisfies 
 $[Q(n)](A) \not\in E(\IQ_q(N/p))$, there
 exists a non-torsion point
$B\in E(\IQbar)$ 
with
$\hat h(B)\le 4\hat h(A)$ and
  \begin{equation*}
\hat h_p(B) 
\ge \frac{\log p}{2p^6}.
  \end{equation*}
\end{lemma}
\begin{proof}
By hypothesis there is $\psi\in\gal{\IQ_q(N)/\IQ_q(N/p)}$ such
that $\psi([Q(n)](A)) \not= [Q(n)](A)$.
It is convenient to identify $\psi$ with its restriction to 
$\IQ(N)$. 
We take the point from  the assertion to be  $B = \psi(A)-A$.
The fact that this  is not a torsion point follows
from  Lemma \ref{lem:kummer_ec}.
Moreover, the parallelogram equality implies
$\hat h(B) \le 2\hat h(\psi(A)) + 2 \hat h(A) = 4\hat h(A)$. 

We now prove the lower bound for $\hat h_p(B)$. The centralizer of
$\psi$ in the global Galois group is the subgroup
\begin{equation*}
  G = \{\sigma\in \gal{\IQ(N)/\IQ};\,\, 
\sigma \psi \sigma^{-1}=\psi\}.
\end{equation*}
For any $\sigma \in G$ we have
\begin{equation*}
\lambda_{\sigma^{-1}v}(B) = 
  \lambda_{v}((\sigma\psi)(A)-\sigma(A))=
  \lambda_{v}((\psi\sigma)(A)-\sigma(A))
\end{equation*}
and $(\psi\sigma)(A)\not=\sigma(A)$. So Lemma \ref{lem:metric1_ec}
applied to $\sigma(A)$ yields
\begin{equation}
\label{eq:lambda_w_lb_ec}
  \lambda_{\sigma^{-1}v}(B) \ge  \frac{\log p}{2(p^2-1)}.
\end{equation}

We will soon show that
 (\ref{eq:lambda_w_lb_ec})
 contributes 
 to the partial height $\hat h_p(B)$ in a significant manner. This will
follow 
since the orbit of $v$ under $G$ is sufficiently large. On the
other hand, if $w$ is any place of $\IQ(N)$ with $w|p$, then 
$\lambda_w(B)\ge 0$ since $E$ has good reduction at $p$. Thus
 \begin{align*}
   \hat h_p(B)  &= \frac{1}{[\IQ(N):\IQ]}\sum_{w|p}
d_w \lambda_w(B) \\
&\ge
\frac{1}{[\IQ(N):\IQ]}\sum_{w\in Gv} d_w\lambda_w(B)  \\
&\ge \frac{1}{[\IQ(N):\IQ]}\frac{\log p}{2(p^2-1)}
d_v \# Gv  \\
&\ge
\frac{\log p}{2(p^2-1)p^4}
 \end{align*}
where in the final inequality
we used the lower bound for $\#Gv$ from
 Lemma \ref{lem:largeorbit}.
\end{proof}

We treat the tamely ramified case $p^2\nmid N$  as in Lemma
\ref{lem:pM}. 

\begin{lemma}
  \label{lem:pM_ec}
We assume  $N\in\IN$ with $p^2 \nmid N$.
If $A\in\IQ(N) \ssm \tors{E}$
 there exists a non-torsion point
$B\in E(\IQbar)$
with
$\hat h(B)\le 2p^{10}\hat h(A)$ and
  \begin{equation*}
\hat h_p(B) \ge \frac{\log p}{2p^6} 
  \end{equation*}
\end{lemma}
\begin{proof}
We may assume $p|N$. 

First, let us
  suppose that some
  conjugate $A'$ of $A$ over $\IQ$
satisfies $[Q(1)](A')=[q(q-1)](A')\not\in E(\IQ_q(N/p))$. Then
Lemma \ref{lem:ptorsion_ec} applied to this conjugate 
provides a non-torsion point $B\in E(\IQbar)$ with 
$\hat h(B) \le 4\hat h(A') = 4\hat h(A)$ 
and $\hat h_p(B) \ge (\log p)/(2p^6)$. 
The first  inequality is clearly more than what we claim.

So we may assume 
$\sigma([q(q-1)](A)) = [q(q-1)](\sigma(A)) \in E(\IQ_q(N/p))$
for all $\sigma\in \gal{\IQ(N)/\IQ}$.
We  apply Lemma \ref{lem:localdescent} 
to the coordinates of $[q(q-1)](A)$ with respect to our Weierstrass
model to find that
$[q(q-1)](A)$ actually lies in $E(\IQ(N/p))$. 
Since $N/p$ is coprime to $p$, 
 Lemma \ref{lem:primetoptorsion_ec}
yields a
non-torsion point $B\in E(\IQbar)$ 
with
\begin{equation*}
  \hat h(B) \le 2(p^2+1)\hat h([q(q-1)](A)) 
= 2(p^2+1)p^4(p^2-1)^2 \hat h(A) \le 2p^{10}\hat h(A)
\end{equation*}
 and $\hat h_p(B) \ge (\log p)/2$. 
\end{proof}

We now 
mimic the argument in Proposition \ref{prop:prebogo} to obtain its
counterpart in the elliptic curve setting.

\begin{proposition}
\label{prop:prebogo_ec}
  Suppose $E$ does not have complex multiplication.
  There exists a prime $p\ge 5$ depending only on $E$ with the
  following property. If $A\in E(\IQ(\tors{E}))\ssm\tors{E}$
 there is a non-torsion point $B\in E(\IQbar)$
with $\hat h(B)\le 20p^{10} \hat h(A)$ and 
\begin{equation*}
   \hat h_p(B) \ge \frac{\log p}{2p^6}. 
\end{equation*}
\end{proposition}
\begin{proof}
 We argue as in the proof of Proposition \ref{prop:prebogo} 
to see that there is a prime $p$  satisfying (P1) and (P2).

There is $N=p^n M$ with $M\in\IN$ coprime to $p$ and $n$ a non-negative
integer  such that $A\in E(\IQ(N))$.
Let $\sigma_2$ be an automorphism coming from Lemma
\ref{lem:center} and let us define 
\begin{equation*}
  C = \sigma_2(A) - [2](A) \in E(\IQ(N)). 
\end{equation*}
The parallelogram equality and other basic properties of the
N\'eron-Tate height give
\begin{equation}
  \label{eq:ntCbound}
\hat h(C) \le 2 \hat h(\sigma_2(A)) + 2 \hat h([2](A))
= 10 \hat h(A). 
\end{equation}

We  fix the least integer $n'\ge 0$ such that 
$C \in E(\IQ(p^{n'}M))$. Of course $n'\le n$. For brevity, we write
$N'=p^{n'}M$.

If 
 $n'\le 1$ then we can apply
 Lemma \ref{lem:pM_ec} to $C$
if we can show that $C$ is
not a torsion point.
If $C$ has finite order then we get
$\hat h(A) = \hat h(\sigma_2(A)) = \hat h([2](A)) = 4\hat h(A)$.
 Hence $\hat h(A) = 0$
which means that $A$ is itself a torsion point by Kronecker's
Theorem. But this contradicts the hypothesis. 
By Lemma \ref{lem:pM_ec} we obtain a non-torsion point
$B\in E(\IQbar)$  with a lower bound for $\hat h_p(B)$ as
in the current lemma. Moreover, $B$  satisfies
\begin{equation*}
  \hat h(B) \le 2p^{10} \hat h(C) \le 20p^{10}\hat h(A). 
\end{equation*}
by (\ref{eq:ntCbound}). This completes the proof if $n'\le 1$.


Now let us assume $n'\ge 2$. 
By minimality of $n'$ and by
 Lemma \ref{lem:localdescent},  there exists $\sigma \in
\gal{\IQ(N)/\IQ}$ with $C'=\sigma(C) \not\in E(\IQ_q(N'/p))$.
We choose  a   witness
 $\psi\in \gal{\IQ_q(N)/\IQ_q(N'/p)}$ testifying
$\psi(C')\not = C'$. 

We set $A'=\sigma(A)$ and obtain
\begin{equation}
\label{eq:Cprop}
  C' = \sigma_2(A') - [2](A') \in E(\IQ(N'))
\end{equation}
because $\sigma_2$ and $\sigma$ commute.

 In order to apply
Lemma \ref{lem:ptorsion_ec} to $C'$ we must
show $[Q(n')](C') = [q](C')\not\in E(\IQ_q(N'/p))$. We suppose the
 contrary is true and derive a contradiction. 
Then 
\begin{equation*}
 \psi(C')-C' = T \in E[q]\ssm\{0\}.
\end{equation*}
We apply $\psi$ to (\ref{eq:Cprop}) and use the fact 
that it commutes with
$\sigma_2$ to obtain
\begin{equation*}
  C'+T = \psi(C') = \sigma_2(\psi(A')) - [2](\psi(A')). 
\end{equation*}
We set $P=\psi(A')-A'$. A short calculation involving
(\ref{eq:Cprop}) gives
$T = \sigma_2(P) - [2](P)$.

As we have often seen,
 $T$ being torsion implies
$\hat h(P)=\hat h(\sigma_2(P)) = \hat h([2](P))=4\hat h(P)$.
Hence $\hat h(P)=0$ and
thus $P$ is a torsion point too. We fix $\widetilde M \in\IN$ coprime to $p$
such that $[\widetilde M](P) \in E[p^\infty]$. 
So $[\widetilde M](P)$ has order dividing $p^n$ by Lemma
\ref{lem:torsionpoint}.
By construction 
 $\sigma_2$ acts on such points by multiplication by $2$, that is
 $\sigma_2([\widetilde M](P)) = [2\widetilde M](P)$.
Therefore, $T=\sigma_2(P) - [2](P)\in E[\widetilde M]$.
We recall $T\in E[q]$ and deduce 
$T=0$ since $q$ and $\widetilde M$ are coprime. This 
contradicts the choice of $T$. So  we must have
$[q](C')\not\in E(\IQ_q(N'/p))$. 

We  may finally apply Lemma \ref{lem:ptorsion_ec}  to $C'$.
It gives us a non-torsion point $B\in E(\IQbar)$ with
\begin{equation*}
  \hat h_p(B) \ge \frac{\log p}{2p^6}
\end{equation*}
and $\hat h(B)\le 4\hat h(C')$. 
But $\hat h(C')=\hat h(C)$ and we recall
(\ref{eq:ntCbound}) to obtain
$\hat h(B) \le 40 \hat h(A)\le 20p^{10}\hat h(A)$, as desired.
\end{proof}

Suppose $B$ and $p$ are as in the previous proposition.
The next lemma relies on  Archimedean and non-Archimedean
equidistribution properties alluded to in the introduction. 
We use it to show that the partial height functions
 $\hat h_\ell(B)$
at places $\ell\not=p$ are negligible if $B$ has small N\'eron-Tate height. 

\begin{lemma}
\label{lem:partialheightlb}
  Let $A_1,A_2,\ldots$ be a sequence of non-torsion points in
  $E(\IQbar)$ with $\lim_{k\rightarrow \infty} \hat h(A_k) = 0$. 
If $\ell$ is a place of $\IQ$, then
\begin{equation*}
  \liminf_{k\rightarrow \infty} \hat h_{\ell}(A_k) \ge 0. 
\end{equation*}
Moreover, if $\ell$ is finite and does not divide the denominator of the
$j$-invariant of $E$, then $\hat h_{\ell}(A_k) \ge 0$
for all $k$.
\end{lemma}
\begin{proof}

We  treat the case $\ell=\infty$ first. Say $A\in E(\IQbar)\ssm\{0\}$, then
\begin{equation*}
  \hat h_\infty(A) = \frac{1}{[\IQ(A):\IQ]} \sum_{\sigma}
\lambda_\infty(\sigma(A))
\end{equation*}
where $\sigma$ runs over all field embeddings
$\IQ(A)\rightarrow\IC$.
Recall that $\lambda_\infty : E(\IC)\ssm\{0\}\rightarrow \IR$ is a
local height function. 
It is continuous, but approaches $+\infty$
 as the argument approaches $0\in
E(\IC)$. So we cannot apply Szpiro, Ullmo, and Zhang's Theorem
to $\lambda_\infty$.
Instead we truncate the local height using a parameter $m\in\IN$
by setting
\begin{equation*}
  \lambda_{\infty,m}(A) = \min\{m,\lambda_\infty(A)\} 
\end{equation*}
for all $A\in E(\IC)\ssm\{0\}$ and $\lambda_{\infty,m}(0)=m$. 
We obtain a continuous function 
$\lambda_{\infty,m}:E(\IC)\rightarrow\IR$ to which
Theorem \ref{thm:SUZ}  applies. So the right-hand side of 
\begin{equation*}
  \hat h_\infty (A_k) \ge \frac{1}{[\IQ(A):\IQ]}
\sum_\sigma \lambda_{\infty,m}(\sigma(A_k))
\end{equation*}
converges to $a_m = \int \lambda_{\infty,m}\mu_{E,\infty}$
as $k\rightarrow \infty$. Therefore, 
$\liminf_{k\rightarrow\infty}\hat
h_\infty(A_k)\ge a_m$.
The functions $\lambda_{\infty,m}$ are pointwise increasing in $m$
with pointwise limit $\lambda_{\infty}$ on $E(\IC)\ssm\{0\}$. 
By the Monotone Convergence Theorem
 $\lambda_{\infty}$ is measurable on $E(\IC)$, its value at $0$
being irrelevant, with
$\lim_{m\rightarrow\infty} a_m =
\int\lambda_{\infty}\mu_{E,\infty}$. The lemma follows for
$\ell=\infty$ if we can show
\begin{equation}
\label{eq:integralzero}
\int\lambda_{\infty}\mu_{E,\infty}=0.
\end{equation}

Indeed, this is  well-known but we provide a short proof. 
We  expressed  $\lambda_\infty$ in
(\ref{eq:localheightinf}) as an infinite series. 
Let $\tau\in \IC$ have positive imaginary part and $q=e^{2\pi i
  \tau}$. 
By the Dominant Convergence Theorem 
it suffices to show that the integral 
over
 \begin{equation*}
   \{z=x+y\tau;\,\, x,y\in [0,1)\} \subset\IC
 \end{equation*}
of each term  vanishes.
Elementary calculus shows
\begin{equation*}
  \int_0^1 b_2(y) dy =  
\int_0^1 \left(y^2-y+\frac 16\right)dy = 0.
\end{equation*}
Suppose $n\ge 1$. Then
\begin{equation*}
  \int_{[0,1)^2} \log|1-q^n e^{\pm 2\pi i(x+y\tau)}|dxdy = 
  \int_{[0,1)^2} \log|e^{\mp 2\pi i x}-e^{2\pi i \tau(n\pm  y)}|dxdy
 = \int_{0}^1 \log \max\{1,|e^{2\pi i \tau(n\pm y)}|\}dy
\end{equation*}
by Jensen's Formula. But $|e^{2\pi i \tau(n\pm y)}| = e^{-2\pi
  \imag{\tau}(n\pm y)}\le 1$ since
$\imag{\tau}>0$ and $y\in [0,1)$.
 So $\int_{[0,1)^2}\log|1-q^n e^{\pm 2\pi i(x+y\tau)}|dxdy = 0$. 
Arguing along similars lines we find
 $\int_0^1 b_2(-y)dy  =\int_{[0,1)^2}\log|1-u|dxdy = 0$. So (\ref{eq:integralzero}) holds true.

Next we treat the case when $\ell$ is a finite place of $\IQ$.

There
exists a finite Galois extension  $F/\IQ$ such that $E$ considered as an elliptic
curve defined over $F$ has either good or split multiplicative reduction at
all finite places. 
Let $p_1,\dots,p_s$ be  precisely the primes dividing the
denominator of the $j$-invariant of $E$. 
By the basic theory of elliptic curves the reduction type of $E$ 
at a finite place $v$ of $F$ is
determined as follows.

 If $v\nmid p_i$ for all
$i$, then $E$ has good reduction at $v$. 
If $v|p_i$ for some $i$, then
$E$ has split multiplication reduction at $v$.

If $\ell$ is not among the $p_i$, then $\lambda_{v}$ is
non-negative for all $v|\ell$ by (\ref{eq:goodreduction}). So
 $\hat h_{\ell}(A) \ge 0$ for all $A\in E(\IQbar)\ssm\{0\}$. 
The second 
statement of
 this lemma follows and the lower bound for the limes inferior is trivial.

Hence let us assume $\ell = p_i$ for some $i$
and let
 $v$ be  a place of $F$ above $\ell$. We fix an
 algebraic closure  $\overline{F_v}$ of $F_v$. 

Suppose $A\in E(\IQbar)\ssm\{0\}$ and
let $K=F(A)$. 
Since $\lambda_v$ is invariant
under the action of $\gal{\overline{F_v}/F_v}$ we have
\begin{equation}
\label{eq:partialheighteq}
  \hat h_\ell(A) = \frac{1}{[F:\IQ]}\sum_{\sigma'} 
 \frac{1}{[K:F]}\sum_{\sigma}  \lambda_v(\sigma(A))
\end{equation}
where $\sigma':F\rightarrow \overline{F_v}$
and $\sigma:K\rightarrow \overline{F_v}$
 run over all field embeddings with $\sigma|_F = \sigma'$. 

Let us fix a field embedding
$\sigma':F\rightarrow\overline{F_v}$. 
We consider 
the Tate uniformization $\overline{F_v}^\times\rightarrow
E(\overline{F_v})$
and let $q_v\in F_v^\times$ denote the associated parameter. 
For any $\sigma$ as above we fix
$u_\sigma \in \overline{F_v}^\times$ with $|q|_v< |u_\sigma|_v \le 1$
in the preimage of $\sigma(A)$. 
Recall that $l_v(A) = \log |u|_v/\log|q|_v + \IZ \in \IR/\IZ$. We define
$\overline{b_2}$ to be   $b_2|_{[0,1)}$ extended
  periodically to $\IR$. 
Then $\overline{b_2}(l_v(A))$ is well-defined and
by (\ref{eq:localheightsmr}) we get
\begin{align}
\label{eq:localheightsmr_lb}
 \lambda_v(\sigma(A))
&=  - \frac{1}{2} 
\overline{b_2}\left(l_v(A)\right) \log|q|_v-\log |1-u_\sigma|_v
\ge
 - \frac{1}{2} 
\overline{b_2}\left(l_v(A)\right) \log|q|_v
\end{align}
since $|1-u_\sigma|_v\le 1$. 

Now we can apply Chambert-Loir's Theorem \ref{thm:CL} to get
\begin{equation*}
  \lim_{k\rightarrow\infty} -\frac 12 \frac{1}{[F(A_k):F]}
 \sum_{\sigma:F(A_k)\rightarrow\overline{F_v}}
\overline{b_2}\left(l_v(A_k)\right) 
= -\frac 12 \int
\overline{b_2}\left(x\right)\mu_{\IR/\IZ}. 
\end{equation*}
The integral on the left is
$\int_0^1 b_2(t)dt$ and hence vanishes,
just  as  in the Archimedean case.
We recall   (\ref{eq:partialheighteq})  and (\ref{eq:localheightsmr_lb})
to derive
\begin{equation*}
  \liminf_{k\rightarrow\infty} \hat h_\ell(A_k) \ge 0.\qedhere
\end{equation*}
\end{proof}

\begin{proof}[Proof of Theorem \ref{thm:main_ec}]
We follow a similar path of argumentation as in the proof of Theorem 
\ref{thm:main}. 
  If $E$ has complex multiplication by an order in an imaginary
  quadratic number field $K$, then $K(\tors{E})$ is an abelian
  extension of $K$. In this case the theorem follows from Baker's
  Theorem 1.1 \cite{Baker:CMBogomolov}. 

So let us suppose that $E$ does not have complex multiplication.

We prove the theorem by contradiction. Let $A_1,A_2,\ldots$ be a
sequence of non-torsion points in $E(\IQ(\tors{E}))$ with $\lim_{k\rightarrow\infty}\hat
h(A_k) = 0 $. Proposition \ref{prop:prebogo_ec} yields a prime $p$ and
a new sequence
$B_1,B_2,\dots$ of non-torsion points in $E(\IQbar)$ with 
$\lim_{k\rightarrow\infty} \hat h(B_k)=0$ 
but
\begin{equation*}
  \hat h_p(B_k) \ge \frac{\log p}{2p^6}.
\end{equation*}
Therefore, 
\begin{equation*}
  \hat h(B_k) \ge \frac{\log p}{2p^6}
+\sum_{\ell\not=p} \hat h_{\ell}(B_k)
\end{equation*}
where $\ell$ ranges over all places of $\IQ$ other than $p$.

By the second statement of Lemma \ref{lem:partialheightlb} we may 
omit all finite places $\ell$ 
that do not appear in the denominator of the $j$-invariant of $E$ in 
the sum on the
right.
So the limes inferior of the 
right-hand side is at least $(\log p)/(2p^6)$ by the first
statement in Lemma \ref{lem:partialheightlb}. 
That of the left-hand side is of course zero and
this is a contradiction.
\end{proof}

\begin{proof}[Proof of Corollary \ref{cor:mwgroup}]
We know from (\ref{eq:torsioninvariance}) that the N\'eron-Tate height
factors through 
$E(\IQ(\tors{E}))/\tors{E}$. 
The square root
$\hat h^{1/2}$  is a positive definite,  homogenous of degree one 
and satisfies the triangle
inequality on this quotient. In short, it is a norm.
The value $0$ is isolated in its image by
 Theorem \ref{thm:main_ec}. So the norm is discrete in the sense of
Zorzitto \cite{Zorzitto}. His theorem implies that
$E(\IQ(\tors{E}))/\tors{E}$ is a free abelian group. 
This group cannot be finitely
generated by the result of Frey and Jarden mentioned in the introduction 
and because $\ab{\IQ}\subset \IQ(\tors{E})$ . 
\end{proof}

A very similar line of thought involving Theorem \ref{thm:main}
implies that the multiplicative group
 $\IQ(\tors{E})^\times$
is isomorphic to $\mu_\infty \oplus \bigoplus_{\IN}\IZ$.

\bibliographystyle{amsplain}
\bibliography{literature}

\vfill
\address{

\noindent
Philipp Habegger,
Johann Wolfgang Goethe-Universit\"at,
Robert-Mayer-Str. 6-8,
60325 Frankfurt am Main,
Germany,
{\tt habegger@math.uni-frankfurt.de}
}
\bigskip
\hrule
\medskip



\end{document}